\documentclass[a4paper,10pt,psamsfonts]{amsart}

\usepackage{amsmath}
\usepackage{amstext}
\usepackage{amsfonts}
\usepackage{amsthm}
\usepackage{amssymb}
\usepackage{dsfont}
\usepackage[bf,SL,BF]{subfigure}
\usepackage{graphicx,epstopdf}

\usepackage{caption}
\usepackage{listings}
\usepackage{hyperref}
\hypersetup{colorlinks}

\usepackage{color}
\usepackage{changes}
\usepackage{lineno}
\usepackage{algorithm}
\usepackage{algorithmic}
\usepackage{epsfig}
\usepackage{pst-all} 
\usepackage{pst-plot} 
\usepackage[space]{grffile}
\usepackage[top=1.0in, bottom=1.0in, left=1.0in, right=1.0in]{geometry}

\newcommand{\triple}{{\vert\kern-0.25ex\vert\kern-0.25ex\vert}}

\theoremstyle{plain}

\newcommand{\be}{\begin{equation}}
\newcommand{\ee}{\end{equation}}

\newtheorem{theorem}{Theorem}[section]
\newtheorem{lemma}[theorem]{Lemma}
\newtheorem{proposition}[theorem]{Proposition}
\newtheorem{corollary}[theorem]{Corollary}
\newtheorem{definition}{Definition}[section]

\newtheorem{assumption}[definition]{Assumption}

\theoremstyle{definition}
\newtheorem{remark}[definition]{Remark}

\begin{document}

\title[Low-variance Sketchy FEM]{Low-Variance Randomised Numerical Linear Algebra for Finite Element Simulation
\footnote{Work made possible through funding support from UKRI EPSRC under grants EP/V028618/1 \& EP/V02860X/1 `RAPID: Real-time Process Modelling and Diagnostics: Powering Digital Factories'.}}

\author[N.~Polydorides]{Nick Polydorides}
\address{Nick Polydorides\\
School of Engineering\\
University of Edinburgh\\
Edinburgh, UK}
\email{n.polydorides@ed.ac.uk}

\author[Y.~Wu]{Yue Wu}
\address{Yue Wu\\
School of Mathematics\\
University of Strathclyde\\
Glasgow, UK}
\email{yue.wu@strath.ac.uk}

\author[H.~Noori]{Hamid Noori}
\address{Hamid Noori\\
Department of Computer Science\\
University of Durham\\
Durham, UK}
\email{hamid.noori@durham.ac.uk}

\author[H.~Vandierendonck]{Hans Vandierendonck}
\address{Hans Vandierendonck\\
Electronics, Electrical Engineering \& Computer Science\\
Queen's University Belfast\\
Belfast, UK}
\email{h.vandierendonck@qub.ac.uk}

\author[R.~Woods]{Roger Woods}
\address{Roger Woods\\
Electronics, Electrical Engineering \& Computer Science\\
Queen's University Belfast\\
Belfast, UK}
\email{r.woods@qub.ac.uk}


\begin{abstract}
We present a low-variance randomised numerical linear algebra approach for  multi-query finite element systems arising from parametric elliptic partial differential equations with applications to digital twins and online model calibration. The method relies on Galerkin subspace projection for reducing the dimensionality, and then combines parameter-oblivious leverage-score Bernoulli sampling with a control variates scheme to yield a reduced-variance `forward' sketch and an invertible `inverse' sketch that are then fused to a single efficient regularised estimator. Effectively, this reduces the computational cost in computing the projected system of equations while preserving the structure, stability, and accuracy of the underlying FEM formulation. We derive probabilistic bounds for the sketching error, invertibility, and estimator variance, and then validate the method on large-scale example problems. The results show that when the parameter fields do not vary too sharply, the synergy of control variates together with the sketch fusion can largely offset the loss incurred by the sub-optimal parameter-oblivious sampling. In this regime, our method achieves substantial savings in time, memory, and communication while maintaining accuracy levels that are acceptable for scientific simulation. 
\end{abstract}

\maketitle
\tableofcontents

\section{Introduction} 

Since its inception, the Finite Element Method (FEM) has been a foundational pillar of computational science, providing a practical, agile and principled framework for simulating complex physical phenomena \cite{zienkiewicz1967finite,Elman_Silvester_Wathen_2014}. Engineering applications with online, multi-query modelling scenarios like those encountered in model calibration and parametric uncertainty quantification require rapid solves of large FEM-derived algebraic systems often on constrained computational resources. To maintain algorithmic scalability and tractability in solving such parametric systems, reduced-basis methods have seen a lot of developments \cite{BennerCohenWillcox}, and although these tend to perform very well, they rely on strong prior assumptions on the parameters. On the other hand, data-driven surrogates like PINNs \cite{Raissi2019Physics} and neural operators \cite{Li2020NeuralOG} bypass the high-dimensionality bottleneck but these require simulation data from the FEM or other numerical solvers to train on, and have outstanding issues in performance generalisation and guarantees \cite{AQdoi:10.1142/S0218202525500125}.

Randomised Numerical Linear Algebra (RNLA) \cite{Martinsson_Tropp_2020} provides a middle-ground alternative preserving, in expectation, the theoretical underpinnings of the FEM, whilst expediting algebraic operations by random subsampling and weighting, also known as sketching. Prior work in this area has addressed the use of RNLA to accelerate Krylov-iteration solvers for linear systems \cite{GowerRichtarik,NakatsukasaTropp,BurkeGuttel} as well as projection-based model reduction with random embeddings \cite{BalabanovNouy,BalabanovNouyII}. Perhaps even more influential to our work is the framework of effective resistances in \cite{DrineasMahoneyER} establishing the link between matrix leverage-scores and resistor networks in the context of graph Laplacian matrices, as indeed that of effective stiffness in \cite{AvronToledo} that extends these ideas to address explicitly FEM systems. In both cases, the matrices involved have a symmetric product structure similar to the one we consider here. Our contribution targets explicitly the multi-query FEM setting by synergistically combining:
\begin{itemize}
\item leverage-score Bernoulli sampling based on probabilities that depend on the model geometry but are oblivious of the FEM parameters,
\item control variates-based sketch variance reduction, and
\item a regularised  formulation that fuses a low-variance forward sketch with an invertible sketch to reduce the mean squared error of the sketched FEM solution. 
\end{itemize} 
The paper is organised as follows: In section 2 we provide a very brief overview of the elliptic Partial Differential Equation (PDE) in consideration, leading to the formation of the parametric FEM linear system, and a more precise definition of our problem's setting. Section 3 provides the groundwork for the subspace projection and explains the rationale for preconditioning the projected equations, before introducing our randomised sketching approach in section 4. Section 5 holds our analysis results where we derive probabilistic bounds on the sketching error, failure probabilities and expressions for the sketch variance, and thereafter in section 6 we present our control variates scheme and an optimisation problem that fuses two sketches into a more efficient sketch, while we also provide a detailed description of our algorithm. Finally, in section 7 we present some numerical results of our approach based on two- and three-dimensional FEM models highlighting the algorithmic  parameters that control its performance. 

\subsection{Notation} 
Scalars and  functions are denoted by lowercase letters. Vectors, such as $\mathbf{x} \in \mathds{R}^n$, are denoted by bold lowercase letters, and matrices like $\mathbf{A} \in \mathds{R}^{n \times m}$, by bold uppercase letters, while their individual entries are written as $x_i$ and $a_{ij}$. The minimum and maximum entries of $\mathbf{x}$ or a diagonal $\mathbf{X}$ are expressed as $x_{\min}$ and $x_{\max}$ respectively. For a matrix $\mathbf{A}$, the $i$th row and column are denoted by $\mathbf{a}_{(i)}$ and $\mathbf{a}^{(i)}$, respectively, $\mathbf{A}^\top$ denotes the transpose, and $\kappa(\mathbf{A})$ its $\ell_2$ condition number. The notation $\mathbf{A} \succeq 0$ means that $\mathbf{A}$ is positive semidefinite, and $\mathrm{Tr}(\mathbf{A})$ denotes the trace of a square matrix $\mathbf{A}$. Unless stated otherwise, $\|\mathbf{x}\|$ denotes the Euclidean norm of $\mathbf{x}$, $\mathrm{diag}(\mathbf{x})$ is a diagonal matrix, while $\mathrm{diag}(\mathbf{A})$ is the column vector based on the main diagonal of $\mathbf{A}$. For matrices, $\|\mathbf{A}\|$ and $\|\mathbf{A}\|_\mathrm{F}$ denote the spectral and Frobenius norms, respectively, while $\sigma_{\max}(\mathbf{A})$ and $\lambda_{\max}(\mathbf{A})$ denote the largest singular value and eigenvalue. For conformably dimensioned matrices $\mathbf{A}$ and $\mathbf{B}$, $\mathbf{A} \circ \mathbf{B}$ denotes the Hadamard product, while identically distributed random matrices are distinguished by subscripts, e.g. $\mathbf{A}_1, \ldots, \mathbf{A}_n$. With $\{N\}$ we denote the set of the first $N$ positive integers. For scalar random variables, $\mathbb{E}[\cdot]$, $\mathbb{V}(\cdot)$, and $\mathbb{C}(\cdot,\cdot)$ denote expectation, variance, and covariance, respectively, while for random vectors or matrices, $\mathbb{E}[\cdot]$ is taken element-wise and $\mathbb{V}(\cdot)$ in the Frobenius-norm sense.

\section{Finite element method preliminaries}

Consider the elliptic PDE with a Dirichlet boundary condition 
\begin{equation}\label{pde}
- \nabla \cdot \vec{p}(\mathbf{x}) \nabla u(\mathbf{x}) = f(\mathbf{x}) \quad  \mathbf{x} \in \Omega, \qquad u(\mathbf{x})= u^\partial(\mathbf{x}) \quad \mathbf{x} \in \partial \Omega,   
\end{equation}
where $\vec{p}$ is a symmetric tensor field in the interior of a domain $\Omega \subset \mathds{R}^{d}$, $d \in \{2,3\}$ with spatial coordinates $\mathbf{x}$. We are interested in finding the solution $u$ under the influence of a forcing function $f$, when the solution on the boundary $\partial \Omega$ is known to be  $u^\partial$. The model equation in \eqref{pde} has a unique solution, in the strong sense, if $f$ is square integrable over the closure of the domain, provided that the boundary of the domain is sufficiently smooth \cite{Elman_Silvester_Wathen_2014}. In the more general, realistic, setting one typically resorts in computing a weak solution instead. In this respect, multiplying \eqref{pde} with a continuous test function $w: \Omega \rightarrow \mathds{R}$ that vanishes on the boundary $\partial \Omega$ and integrating by parts yields the weak form of the PDE   
\begin{equation}\label{weakform}
\int_\Omega \vec{p}(\mathbf{x}) \nabla u(\mathbf{x}) \cdot \nabla w(\mathbf{x}) \, \mathrm{d}\mathbf{x} - \int_\Omega f(\mathbf{x}) \, w(\mathbf{x}) \, \mathrm{d}\mathbf{x} = \int_{\partial \Omega} w(\mathbf{x}) \, \vec{p}(\mathbf{x}) \nabla u(\mathbf{x}) \cdot \mathbf{\hat n} \, \mathrm{d}s,  
\end{equation}
where $\mathbf{\hat n}$ is the outward unit normal on the boundary. To make the problem computationally tractable, the continuous functions in the weak form are approximated by interpolation using functions of compact support. To enable this reduction to finite dimension, the domain geometry is discretised to a mesh (grid) $\Omega_h(n_n,n_e)$ comprising $n_n$ nodes  connected in $n_e$ linear elements. This allows to approximate $\vec{p}$ and $f$ as piecewise constant functions on elements and the solution as a piecewise linear function with $n$ degrees of freedom  
\begin{equation}\label{dofsplit}
u(\mathbf{x}) \approx \sum_{i=1}^{n_n} u_i \varphi_i(\mathbf{x}) = \sum_{i=1}^{n} u_i \varphi_i(\mathbf{x}) + \sum_{i=n+1}^{n_n} u^\partial_i \varphi_i(\mathbf{x})
\end{equation}
conforming the boundary condition in \eqref{pde} where $\{\varphi_i\}_{i=1}^{n_n}: \Omega_h \rightarrow \mathds{R}$ are linear interpolation functions. Setting $w=u$ and $u=u^\partial$ in the boundary integral of  \eqref{weakform} and imparting \eqref{dofsplit} we arrive at a system of $n$ linear algebraic equations in $N \doteq dn_e$ parameters and $n$ unknowns
\begin{equation}\label{fempde}
\mathbf{A}(\mathbf{P})\mathbf{u} = \mathbf{d}(\mathbf{P}) 
\end{equation}
where  
\begin{equation}\label{AP_and_dP}
\mathbf{A}(\mathbf{P})= \mathbf{D}^\top \mathbf{P} \mathbf{D}, \qquad \mathbf{d}(\mathbf{P}) \doteq \mathbf{f} - \mathbf{A}_\partial (\mathbf{P}) \mathbf{u}^\partial, \qquad \mathbf{A}_\partial(\mathbf{P}) = \mathbf{D}^\top \mathbf{P} \mathbf{D}_\partial,
\end{equation}
with $\mathbf{P} \in \mathds{R}^{N \times N}$ a bounded positive diagonal matrix and $\mathbf{D} \in \mathds{R}^{N \times n}$ tall sparse matrix with full column rank, resulting in $\mathbf{A}(\mathbf{P})$ being symmetric positive definite (SPD). On the right hand side, $\mathbf{A}_\partial(\mathbf{P}) \in \mathds{R}^{n \times m}$ and $\mathbf{D}_\partial \in \mathds{R}^{N \times m}$ are tall sparse matrices with full column rank $m \doteq n_n-n$, whilst vector $\mathbf{f} \in \mathds{R}^n$ depends on the forcing term. Although $\mathbf{A}(\mathbf{P})$ is by construction SPD, for certain choices of $\mathbf{P}$ and or $\mathbf{D}$, the later of which relates to the regularity of the mesh elements, it may become ill-conditioned \cite{KANNAN201479}. We thus introduce the following assumption regarding the solvability of the high-dimensional FEM system. 

\vskip 10 pt
\begin{assumption}\label{HDsolvable}
For any admissible $\mathbf{P} \succ 0$ the solution of \eqref{fempde}
\begin{equation}\label{solHD}
\mathbf{u}^\star = \mathbf{A}^{-1}(\mathbf{P}) \bigl ( \mathbf{f} - \mathbf{A}_\partial(\mathbf{P}) \mathbf{u}^\partial \bigr )
\end{equation}
exists and has finite norm \cite{Elman_Silvester_Wathen_2014}. In double-precision floating-point arithmetic, this implies that the condition number of the matrix is bounded by $ \kappa(\mathbf{A}(\mathbf{P})) \ll 10^{16} $. 
\end{assumption}
Under standard assumptions, solving \eqref{fempde} has complexity $\mathcal{O}(n^3)$ with direct methods, whereas non-preconditioned iterative methods require about $ \mathcal{O}(\sqrt{\kappa(\mathbf{A})}\,\mathrm{nnz}(\mathbf{A}))$ operations \cite{Saad}. In the dimensions $(n,N)$ of interest in this work, both direct and iterative algorithms, even when combined with preconditioning, become prohibitively expensive in the multi-query setting. We conclude this introductory section by giving a more precise definition of the problem under consideration.

\begin{definition}
Given a stream of positive and bounded parameter diagonals $\mathbf{P}^{(1)}, \mathbf{P}^{(2)},\ldots,$ approximate the solutions $\mathbf{u}(\mathbf{P}^{(1)}), \mathbf{u}(\mathbf{P}^{(2)}), \ldots,$ of \eqref{fempde} for a fixed $\mathbf{f}$, in a memory, communication and time-efficient manner.
\end{definition}

\section{Subspace projection}

The first stage of our approach introduces a low-dimensional structure for the sought solution vectors in the form of a subspace projection. Owing to the smoothness of the solution to the underlying PDE \eqref{pde}, a characteristic feature of steady-state diffusion problems, we seek to exploit the resulting low-dimensional structure to reduce the number of degrees of freedom in the discrete model. Although such a structure is not immediately apparent in the high-dimensional system \eqref{fempde}, it can be exploited using projection-based model order reduction techniques such as Proper Orthogonal Decomposition (POD) \cite{BennerCohenWillcox}. In effect, we can fix a priori a low-dimensional subspace spanned by $s \ll n$ orthonormal basis functions that captures, on average, the main features of the solutions $\mathbf{u}(\mathbf{P})$ for arbitrary choices of $\mathbf{P}$. The advantage of using the same set of basis vectors for all parameter choices is that it allows for any compute-heavy operations to be performed offline, even though such bases cannot be made optimal for the individual parameter choices. The projection effectively induces an approximation error, which sets a lower bound on the overall error of our approach. To see the effect of the subspace projection approximation let us introduce the subspace 
\begin{equation}
\mathds{S} \doteq \bigl \{ \boldsymbol{\Phi} \mathbf{w} \, \cdot \, \mathbf{w} \in \mathds{R}^s \bigr \},
\end{equation}
where $\boldsymbol{\Phi} \in \mathds{R}^{n \times s}$ is a tall matrix holding the orthonormal basis vectors in its columns and consider the decomposition of the high-dimensional solution of \eqref{solHD} into the component that lies into the subspace and that which is orthogonal to it as $\mathbf{u}^\star = \boldsymbol{\Pi} \mathbf{u}^\star + (\mathbf{I} - \boldsymbol{\Pi}) \mathbf{u}^\star$, where $\boldsymbol{\Pi}: \mathds{R}^n \rightarrow \mathds{S}$ denotes the orthogonal projection operator. Substituting into the linear system \eqref{fempde} we have
\begin{equation}\label{projmodel}
\mathbf{d}(\mathbf{P}) = \mathbf{A}(\mathbf{P}) \boldsymbol{\Pi}\mathbf{u}^\star + \boldsymbol{\epsilon}, 
\end{equation}
where $\boldsymbol{\epsilon} \doteq \mathbf{A}(\mathbf{P}) \bigl (\mathbf{I} - \boldsymbol{\Pi} \bigr )\mathbf{u}^\star$ is the subspace approximation-induced data discrepancy. The point $\boldsymbol{\Pi}\mathbf{u}^\star$ coincides with $\boldsymbol{\Phi}\mathbf{w}^\star$ for an optimal $\mathbf{w}^\star=\boldsymbol{\Phi}^\top \mathbf{u}^\star$. If the residual $\boldsymbol{\epsilon}$ is orthogonal to the subspace $\mathds{S}$ then substituting  $\boldsymbol{\Pi}\mathbf{u}=\boldsymbol{\Phi}\mathbf{w}$ and performing a Galerkin projection of the equation \eqref{projmodel} onto $\mathds{S}$ leads to the subspace-projected FEM system
\begin{equation}\label{projlin}
\boldsymbol{\Phi}^\top \mathbf{A}(\mathbf{P}) \boldsymbol{\Phi} \, \mathbf{w} = \boldsymbol{\Phi}^\top \mathbf{d}(\mathbf{P}).
\end{equation}
Let $\mathbf{w} \in \mathds{R}^s$ be the solution of \eqref{projlin}, then 
$$
\mathbf{w}^\star -\mathbf{w} = \boldsymbol{\Phi}^\top \mathbf{u}^\star - \bigl (\boldsymbol{\Phi}^\top \mathbf{A}(\mathbf{P}) \boldsymbol{\Phi} \bigr )^{-1} \boldsymbol{\Phi}^\top \mathbf{d}(\mathbf{P}),
$$
and from the triangle inequality we can deduce the overall projection-induced error
\begin{equation}\label{projinderr}
\|\mathbf{u}^\star - \boldsymbol{\Phi}\mathbf{w}\| \leq \|(\mathbf{I} - \boldsymbol{\Pi})\mathbf{u}^\star \|+ \|\boldsymbol{\Phi}^\top \mathbf{u}^\star - \mathbf{w}\|,   
\end{equation}
revealing that, even in solving \eqref{projlin} exactly, the solution $\boldsymbol{\Phi}\mathbf{w}$ will still be away from the high-dimensional FEM solution $\mathbf{u}^\star$ by a distance that depends not only on the distance of $\mathbf{u}^\star$ from $\mathds{S}$ but also on that of the low-dimensional $\mathbf{w}$ from the projection of $\mathbf{u}^\star$ on $\mathds{S}$. 

\subsection{Preconditioning}

When using high-dimensional unstructured meshes, the FEM coefficients (stiffness) matrix $\mathbf{A}(\mathbf{P})$ may become ill-conditioned. This is known to occur  when the variation in the values of the diagonal $\mathbf{P}$ spans several orders of magnitude \cite{VavasisWildCoeff}, or when the mesh of the domain's geometry includes some elements with very obtuse or acute angles, which in turn causes some of the elements of the gradients matrix  $\mathbf{D}$ to be extremely large or almost zero \cite{KANNAN201479}. It is thus advantageous to precondition the projected equations in order to mitigate any undesired geometry effects.  

As the subspace projection basis is kept fixed for all parameter choices, using a thin Singular Value Decomposition (SVD) $\mathbf{D} \boldsymbol{\Phi}=\mathbf{U}\boldsymbol{\Sigma}\mathbf{V}^\top$,  the coefficients matrix of the projected system \eqref{projlin} can be further developed as
\begin{equation}\label{phitAphi}
\boldsymbol{\Phi}^\top \mathbf{A}(\mathbf{P}) \boldsymbol{\Phi} = (\mathbf{D} \boldsymbol{\Phi})^\top \mathbf{P} \, (\mathbf{D} \boldsymbol{\Phi}) = \mathbf{V} \boldsymbol{\Sigma}^\top \, ( \mathbf{U}^\top \mathbf{P} \mathbf{U} ) \, \boldsymbol{\Sigma} \mathbf{V}^\top, 
\end{equation}
allowing to recast the projected system in a preconditioned form as
\begin{equation}\label{projY}
\mathbf{Y}(\mathbf{P}) \mathbf{v} = \mathbf{q}(\mathbf{P}),
\end{equation}
for a symmetric coefficients matrix 
\begin{equation}\label{targetY}
    \mathbf{Y}(\mathbf{P}) \doteq \mathbf{U}^\top \mathbf{P} \mathbf{U} 
\end{equation}
through a change of variables $\mathbf{v} \doteq \boldsymbol{\Sigma}\mathbf{V}^\top \mathbf{w}$,
and a right hand side vector
\begin{equation}\label{qP}
\mathbf{q}(\mathbf{P}) = \boldsymbol{\Sigma}^{-1} \mathbf{V}^\top \boldsymbol{\Phi}^\top \mathbf{d}(\mathbf{P}). 
\end{equation}
By virtue of assumption \ref{HDsolvable}, $\mathbf{D}$ has full column rank and since $\boldsymbol{\Phi}$ has orthonormal columns, $\mathrm{rank}(\mathbf{D}\boldsymbol{\Phi})=s$ which implies that $\mathbf{\Sigma} \succ 0$, and the mapping between $\mathbf{v}$ and $\mathbf{w}$ is bijective. Further, as the $\ell_2$ condition number of the coefficients matrix is important for the numerical stability of the solution, it is prudent to examine those of the original \eqref{phitAphi} and the preconditioned \eqref{targetY} projected matrices. Since both $\mathbf{Y}(\mathbf{P})$ and  $\mathbf{V}\boldsymbol{\Sigma}^\top \mathbf{Y}(\mathbf{P}) \boldsymbol{\Sigma}\mathbf{V}^\top$ are SPD, then as $\mathbf{V}$ is orthogonal then $\kappa(\mathbf{V}\boldsymbol{\Sigma}^\top \mathbf{Y}(\mathbf{P}) \boldsymbol{\Sigma}\mathbf{V}^\top) = \kappa(\boldsymbol{\Sigma}^\top \mathbf{Y}(\mathbf{P}) \boldsymbol{\Sigma})$. On the other hand, since $\mathbf{U}$ has orthonormal columns $\kappa(\mathbf{Y}(\mathbf{P})) \leq \kappa(\mathbf{P})=p_{\max}/p_{\min}$. In the special case of $\mathbf{P}=\mathbf{I}$,
$\mathbf{Y}(\mathbf{I})$ is perfectly conditioned while $\kappa(\boldsymbol{\Sigma}^\top \mathbf{Y}(\mathbf{I}) \boldsymbol{\Sigma}) =  \bigl (\sigma_{\max}/ \sigma_{\min} \bigr )^2$  can be arbitrarily ill-conditioned. The general case however is more subtle, as it depends on how the eigenstructure of $\mathbf{Y}(\mathbf{P})$ interacts with the diagonal scaling induced by $\boldsymbol{\Sigma}$. 

\vskip 10 pt

\begin{remark}\label{precond}

Consider $\mathbf{C} \doteq \mathbf{Y}^{\frac 1 2}(\mathbf{P}) \boldsymbol{\Sigma} \mathbf{Y}^{-\frac 1 2}(\mathbf{P})$ through which the projected FEM matrix is expressed as
$$
\boldsymbol{\Sigma}^\top \mathbf{Y}(\mathbf{P}) \boldsymbol{\Sigma} = \mathbf{Y}^{\frac 1 2}(\mathbf{P}) \mathbf{C}^\top \mathbf{C} \mathbf{Y}^{\frac 1 2} (\mathbf{P}).
$$
Then $\sigma_{\min}(\mathbf{C})^2 \mathbf{I} \preceq \mathbf{C}^\top \mathbf{C} \preceq \sigma_{\max} (\mathbf{C})^2 \mathbf{I}$, and thus $\sigma_{\min} (\mathbf{C})^2\,\mathbf{Y}(\mathbf{P}) \;\preceq\; \boldsymbol{\Sigma}^\top \mathbf{Y}(\mathbf{P}) \boldsymbol{\Sigma} \;\preceq\; \sigma_{\max}(\mathbf{C})^2 \,\mathbf{Y}(\mathbf{P})$. Moreover, for SPD matrices $ \mathbf{Y}(\mathbf{P})$ and $\boldsymbol{\Sigma}^\top \mathbf{Y}(\mathbf{P}) \boldsymbol{\Sigma}$ this implies
$$
\frac{\kappa\!\left(\mathbf{Y}(\mathbf{P})\right)}{\kappa(\mathbf{C})^{2}}
\leq
\kappa\!\left(\boldsymbol{\Sigma}^{\top}\mathbf{Y}(\mathbf{P})\boldsymbol{\Sigma}\right)
\leq
\kappa(\mathbf{C})^{2}\kappa\!\left(\mathbf{Y}(\mathbf{P})\right).
$$
Thus $\kappa(\mathbf{C}) $ may be interpreted as a measure of the extent to which the diagonal scaling induced by $\boldsymbol{\Sigma}$ distorts the conditioning of the projected system: when $\kappa(\mathbf{C}) \approx 1$, this effect is negligible, whereas for large $ \kappa(\mathbf{C})$ it may be substantial, thereby conferring an advantage to the preconditioned system.
\end{remark}
Remark \ref{precond} shows that the although the presence of $\boldsymbol{\Sigma}$ as in the original projected equations can improve the conditioning,  heuristically this occurs when directions associated with large eigenvalues of $\mathbf{Y}(\mathbf{P})$ are weighted weakly by $\boldsymbol{\Sigma}$, while those associated with small eigenvalues are weighted more strongly, thereby driving $\kappa(\mathbf{C})$ closer to $1$. In general, however, the effect of $\boldsymbol{\Sigma}$ is not necessarily beneficial, and it may worsen conditioning by as much as a factor of $\kappa(\boldsymbol{\Sigma})^2$. The reformulation in \eqref{projY} removes this additional potential source of ill-conditioning and separates the conditioning of $\mathbf{Y}(\mathbf{P})$ from the interaction between $\boldsymbol{\Sigma}$ and the eigenstructure of $\mathbf{Y}(\mathbf{P})$.

\section{Randomised sketching}

Assembling the low-dimensional system \eqref{projY}, requires computing the $s \times s$ symmetric matrix $\mathbf{Y}(\mathbf{P})$ and the $s \times 1$ vector $\mathbf{q}(\mathbf{P})$, both of which involving $N$-dimensional matrix products. To mitigate the computational complexity when $N$ is extremely large, we appeal to the RNLA framework to construct estimators $\hat{\mathbf{Y}}(\mathbf{P})$ and  $\hat{\mathbf{q}}(\mathbf{P})$ such that $\hat{\mathbf{v}} = \hat{\mathbf{Y}}^{-1}(\mathbf{P}) \hat{\mathbf{q}}(\mathbf{P})$ has a small mean squared error. The sought $\hat{\mathbf{Y}}(\mathbf{P})$ must be invertible with high probability, and such a sketch can be designed based on the spectral norm criterion \cite{Drineas2012Fast} 
\begin{equation}\label{sketchcriterionY}
\hat{\mathbf{Y}}^\star(\mathbf{P}) = \arg \min_{\hat{\mathbf{Y}} \in \mathds{R}^{s \times s}} \| \hat{\mathbf{Y}}^{-1}(\mathbf{P}) \mathbf{Y}(\mathbf{P}) - \mathbf{I}\|.
\end{equation}
Regarding the sketch of the right hand side vector, notice that through a thin SVD of the $N \times m$ matrix $\mathbf{D}_\partial = \mathbf{U}_\partial \boldsymbol{\Sigma}_\partial \mathbf{V}_\partial^\top$ this becomes
\begin{equation}\label{qPsplit}
    \mathbf{q}(\mathbf{P}) = \bigl ( \mathbf{V} \boldsymbol{\Sigma} \bigr )^{-1} \boldsymbol{\Phi}^\top \mathbf{f} - \mathbf{U}^\top \mathbf{P} \mathbf{U}_\partial \, \boldsymbol{\Sigma}_\partial \mathbf{V}_\partial^\top  \mathbf{u}^\partial, 
\end{equation}
which for any instance of the parameter matrix, can be estimated by sketching the $s \times m$ matrix
\begin{equation}
\mathbf{Z}(\mathbf{P}) \doteq \mathbf{U}^\top \mathbf{P}\mathbf{U}_\partial.
\end{equation}
As $\mathbf{q}(\mathbf{P})$ is linear in $\mathbf{Z}(\mathbf{P})$, this vector is optimally approximated when the sketch of $\mathbf{Z}(\mathbf{P})$ is constructed based on the minimum variance criterion \cite{DM16}
\begin{equation}\label{Zfrocriterion}
\hat{\mathbf{Z}}^\star = \arg \min_{\hat{\mathbf{Z}} \in \mathds{R}^{s \times m}} \|\hat{\mathbf{Z}}(\mathbf{P}) - \mathbf{Z}(\mathbf{P})\|_\mathrm{F},
\end{equation}
ultimately leading to a sketch-approximated projected right hand side
\begin{equation}\label{sketchcriteriond}
\hat{\mathbf{q}}(\mathbf{P}) = \bigl ( \mathbf{V} \boldsymbol{\Sigma} \bigr )^{-1} \boldsymbol{\Phi}^\top \mathbf{f} - \hat{\mathbf{Z}}(\mathbf{P}) \, \boldsymbol{\Sigma}_\partial \mathbf{V}_\partial^\top  \mathbf{u}^\partial,
\end{equation}
where vector $\mathbf{u}^\partial$ is typically low-dimensional given that its support is restricted to the boundary of the domain. Furnished with the two matrix sketches satisfying \eqref{sketchcriterionY} and \eqref{Zfrocriterion} respectively, we can readily obtain a \emph{single sketch estimator} of the projected equation solution \cite{LUNG2020112933}
\begin{equation}\label{wandv}  
\hat{\mathbf{w}} = \mathbf{V} \boldsymbol{\Sigma}^{-1} \hat{\mathbf{v}}, \quad \hat{\mathbf{v}} = \hat{\mathbf{Y}}^{-1}(\mathbf{P}) \hat{\mathbf{q}}(\mathbf{P}).
\end{equation}
Exploiting the positive definiteness of the diagonal $\mathbf{P}$ we can cast the desired matrices as weighted outer products as
\begin{equation}\label{multideterm}
\mathbf{Y}(\mathbf{P}) = (\mathbf{P}^{\frac 1 2} \mathbf{U})^\top (\mathbf{P}^{\frac 1 2} \mathbf{U}), \quad \text{and} \quad \mathbf{Z}(\mathbf{P}) = (\mathbf{P}^{\frac 1 2} \mathbf{U})^\top (\mathbf{P}^{\frac 1 2} \mathbf{U}_\partial),
\end{equation}
and thereafter subsample their rows according to sampling distributions that conform the criteria in \eqref{sketchcriterionY} and \eqref{Zfrocriterion}. To derive these distributions we recall the definition of the leverage scores of a matrix. 

\begin{definition}\label{Leverages}
For a real $N \times s$ matrix $\mathbf{A}=\mathbf{U}_{\mathbf{A}}\boldsymbol{\Sigma}_{\mathbf{A}}\mathbf{V}_{\mathbf{A}}^\top$ with $N > s$ and $\mathrm{rank}(\mathbf{A})=s$, the statistical leverage scores associated with its rows are given by \cite{Drineas2012Fast}
\begin{equation}\label{ell_i_A}
\ell_i(\mathbf{A}) \doteq \mathbf{a}_{(i)}(\mathbf{A}^\top \mathbf{A})^{-1} \mathbf{a}_{(i)}^\top = {\mathbf{u}_{\mathbf{A}}}_{(i)}{\mathbf{u}_{\mathbf{A}}}_{(i)}^\top = \bigl \|{\mathbf{u}_{\mathbf{A}}}_{(i)} \bigr \|^2, \quad \forall i \in \{N\},
\end{equation}
and these satisfy $0 < \ell_i(\mathbf{A}) < 1$ and $\sum_{i=1}^N \ell_i(\mathbf{A}) = s$. 
\end{definition}

Let $\boldsymbol{\xi}(\mathbf{U})$ be a discrete sampling distribution over the index set $\{N\}$ and consider independent Bernoulli variables $\gamma_1,\ldots,\gamma_N$with non-identical probabilities 
\begin{equation}\label{Berndef}
   \eta_i = \mathbb{P}(\gamma_i = 1) = \min \bigl \{1,c \, \xi_i( \mathbf{U}) \bigr \} > 0, \quad \text{with} \quad \xi_i(\mathbf{U}) = \frac 1 s \ell_i(\mathbf{U}) \; \text{for all} \, i \in \{N\}
\end{equation}
and a positive integer $c = \mathbb{E} \bigl [\sum_{i=1}^N \eta_i \bigr ]$ that upper bounds the expected sample size, even though the actual number of samples drawn is random. This allows to define a random sketching diagonal matrix $\mathbf{S} \in \mathds{R}^{N \times N}$ 
\begin{equation}\label{sdiag}
    \mathbf{S} = \mathrm{diag} \Bigl ( \frac{\gamma_1}{\sqrt{\eta_1}}, \ldots, \frac{\gamma_N}{\sqrt{\eta_N}} \Bigr ),    
\end{equation}
leading to a sketch estimator 
\begin{equation}\label{invsketch}
\hat{\mathbf{Y}}(\mathbf{P}) = (\mathbf{S}\mathbf{P}^{\frac 1 2} \mathbf{U})^\top (\mathbf{S}\mathbf{P}^{\frac 1 2} \mathbf{U}).
\end{equation}
This Bernoulli-based sketching is equivalent to sampling the rows of $\mathbf{P}^{\frac 1 2} \mathbf{U}$ independently without replacement, since its rows can be sampled at most once. Crucially however, this sampling strategy is, by choice, oblivious to the parameters as the sampling probabilities depend on the leverage scores of $\mathbf{U}$ rather than those of the sampled $\mathbf{P}^{\frac 1 2}\mathbf{U}$. This is because in the multi-query scenario we focus on, computing the exact leverage scores for each $\mathbf{P}$ is computationally more expensive than performing the multiplication in \eqref{multideterm} deterministically. Hence instead of $\boldsymbol{\xi}(\mathbf{P}^{\frac 1 2}\mathbf{U})$ for each choice of $\mathbf{P}$, we settle by necessity to the suboptimal  $\boldsymbol{\xi}(\mathbf{U})$ and the associated Bernoulli probabilities $\boldsymbol{\eta}$ that can be precomputed offline. 
This approximation of the optimal leverage score sampling comes at the expense of a higher sketching error, even though it retains some partial information of the sampled matrix through $\mathbf{U}$.  

In regards to constructing the sketched estimator for $\mathbf{Z}(\mathbf{P})$ as in \eqref{multideterm} with minimal variance as per the Frobenius norm criterion \eqref{Zfrocriterion} we adapt a classical algorithm for randomised matrix multiplication \cite{DrineasKannanMahoney2006} and set 
\begin{equation}\label{fwdsketch}
\hat{\mathbf{Z}}(\mathbf{P}) \doteq \bigl (\mathbf{M} \mathbf{P}^{\frac 1 2} \mathbf{U} \bigr )^\top \bigl ( \mathbf{M} \mathbf{P}^{\frac 1 2} \mathbf{U}_\partial \bigr ),
\end{equation}
where $\mathbf{M} \in \mathds{R}^{N \times N}$ is a random sparse diagonal matrix 
\begin{equation}\label{checksdiag}
\mathbf{M} \doteq \mathrm{diag} \Bigl ( \frac{\tilde{\gamma}_1}{\sqrt{\tilde{\eta}_1}}, \ldots, \frac{\tilde{\gamma}_N}{\sqrt{\tilde{\eta}_N}} \Bigr ),    
\end{equation}
and $\tilde{\gamma}_1, \ldots, \tilde{\gamma}_N$ are independent Bernoulli random variables with probabilities
\begin{align}\label{defnzeta}
\tilde{\eta}_i = \mathbb{P}(\tilde{\gamma}_i=1) =\min \bigl \{1, c \, \zeta_i (\mathbf{U},\mathbf{U}_\partial) \bigr \} > 0, \; \text{with} \; \zeta_i(\mathbf{U},\mathbf{U}_\partial) = \frac{\ell_i(\mathbf{U}) \, \ell_i(\mathbf{U}_\partial)}{\sum_{j=1}^N \ell_j(\mathbf{U}) \ell_j(\mathbf{U}_\partial)} \; \text{for all} \, i \in \{N\}.
\end{align}
Unlike \cite{DrineasKannanMahoney2006} however, we opt not to form the optimal distribution but rather sample with suboptimal probabilities that do not depend on the parameters thus allowing to use the same  $\tilde{\boldsymbol{\eta}}$-based sketching matrix $\mathbf{M}$ for every parameter query, which in turn samples  same index rows from $\mathbf{P}^{\frac 1 2} \mathbf{U}$ and  $\mathbf{P}^{\frac 1 2} \mathbf{U}_\partial$ independently and without replacement.

With these sketching mechanisms in place, we can now introduce the principal random objects underpinning our developments. Foremost among them is an ensemble of independent and identically distributed (iid) sketches. Let $\nu>1$ be a positive integer, then define the \emph{multi-sketch average estimators}
\begin{equation}\label{sketchave}
\bar{\mathbf{Y}}(\mathbf{P}) = \frac 1 \nu \sum_{t=1}^{\nu} \hat{\mathbf{Y}}_t(\mathbf{P}), \quad \text{and} \quad \bar{\mathbf{Z}}(\mathbf{P}) = \frac 1 \nu \sum_{t=1}^{\nu} \hat{\mathbf{Z}}_t(\mathbf{P}),    
\end{equation}
where each ensemble involves $\nu$ iid sketches according to the probabilities $\boldsymbol{\eta}$ and $\boldsymbol{\tilde{\eta}}$ respectively, i.e. for $t=1,\ldots, \nu$ we have $\hat{\mathbf{Y}}_t = \mathbf{U}^\top \mathbf{P}^{\frac 1 2}\mathbf{S}_t^\top \mathbf{S}_t \mathbf{P}^{\frac 1 2} \mathbf{U}$ where $\mathbf{S}_1, \ldots, \mathbf{S}_t, \ldots \mathbf{S}_{\nu}$ are iid copies of the sketching matrix in \eqref{sdiag} and similarly for $\hat{\mathbf{Z}}_t$.
This averaging yields the multi-sketch solution estimator
\begin{equation}\label{multisketchave}
\bar{\mathbf{w}} = (\boldsymbol{\Sigma} \mathbf{V}^\top \bigr )^{-1} \bar{\mathbf{v}}, \quad \bar{\mathbf{v}} = \bar{\mathbf{Y}}^{-1}(\mathbf{P}) \bar{\mathbf{q}}(\mathbf{P}),   
\end{equation}
where $\bar{\mathbf{q}}(\mathbf{P})$ is obtained from \eqref{sketchcriteriond} by swapping the single sketch estimator with the $\nu$-sketch average. For the sake of clarity we omit the $\nu$-dependence from the multi-sketch averages while we also fix the expected per-sketch sample budget at $c$. Subject to $\bar{\mathbf{Y}}(\mathbf{P})$ being invertible, it is easy to see that with a total of about $\nu c$ independent samples, $\bar{\mathbf{w}}$ is more efficient than the single sketch estimator in \eqref{wandv}.


\section{Error analysis}\label{erroranalysis}

\subsection{On the coefficients matrix sketch}

We begin our analysis by investigating the impact of suboptimal probabilities on the multi-sketch average estimator $\bar{\mathbf{Y}}(\mathbf{P})$, focusing on bounding the probability of having a non-invertible sketch, and thereafter address its bias and variance induced errors. Before we embark on this task we must invoke an intermediate lemma. 

\begin{lemma}\label{normAmI} \cite{LUNG2020112933}
For $\mathbf{P} \succ 0$, let the thin SVD $\mathbf{P}^{\frac 1 2} \mathbf{U} = \mathbf{U}_y \boldsymbol{\Sigma}_y \mathbf{V}_y^\top$ and a non-negative diagonal sketching matrix $\mathbf{S}$ such that  $\|\mathbf{U}_y^\top \mathbf{S}^\top \mathbf{S} \mathbf{U}_y - \mathbf{I} \| \leq \epsilon$ for $0 < \epsilon < 1$, then \begin{equation}\label{lem1}
\frac{\epsilon}{1 + \epsilon} \leq \|\bigl ( \mathbf{U}_y^\top \mathbf{S}^\top \mathbf{S} \mathbf{U}_y \bigr )^{-1} - \mathbf{I}\| \leq \frac{\epsilon}{1 - \epsilon}.
\end{equation}
\end{lemma}

A sketched estimator \eqref{invsketch} that satisfies the conditions of lemma \ref{normAmI} incurs a sketching error that is bounded in spectral norm as
\begin{align} \nonumber
\|\hat{\mathbf{Y}}^{-1}(\mathbf{P}) \mathbf{Y}(\mathbf{P}) - \mathbf{I}\| & = \|\bigl (\mathbf{V}_y \boldsymbol{\Sigma}_y \mathbf{U}_y^\top \mathbf{S}^\top \mathbf{S} \mathbf{U}_y \boldsymbol{\Sigma}_y \mathbf{V}_y^\top \bigr )^{-1}\mathbf{V}_y \boldsymbol{\Sigma}_y^2 \mathbf{V}_y^\top - \mathbf{I}\|\\ \nonumber
& = \|\boldsymbol{\Sigma}_y^{-1} \bigl ( \mathbf{U}_y^\top \mathbf{S}^\top \mathbf{S} \mathbf{U}_y \bigr )^{-1}\boldsymbol{\Sigma}_y -  \boldsymbol{\Sigma}_y^{-1}\boldsymbol{\Sigma}_y\|\\ \nonumber
& \leq \sqrt{\kappa(\mathbf{Y}(\mathbf{P}))} \|(\mathbf{U}_y^\top \mathbf{S}^\top \mathbf{S} \mathbf{U}_y)^{-1} - \mathbf{I}\|\\ \label{boundfactor}
& \leq \frac{\epsilon}{1 - \epsilon} \sqrt{\kappa(\mathbf{Y}(\mathbf{P}))} \|\mathbf{U}_y^\top \mathbf{S}^\top \mathbf{S} \mathbf{U}_y - \mathbf{I}\|,
\end{align}
revealing the impact of the condition number of the matrix being sketched on the sketching error. 

Now consider the multi-sketch average estimator $\bar{\mathbf{Y}}(\mathbf{P})$ as in \eqref{sketchave} formed by $\nu$ iid sketches, based on the same $\mathbf{P}$-oblivious sampling probabilities and a sketch budget $c$. The last inequality in \eqref{boundfactor} allows to bound the error of the multi-sketch average estimator in probability. 

\begin{theorem}\label{concentave}

Let $\mathbf{P} \in \mathds{R}^{N \times N}$ be a positive diagonal and let $\mathbf{U} \in \mathds{R}^{N \times s}$ have orthonormal columns. Define $\mathbf{Y}(\mathbf{P}) = \mathbf{U}^\top \mathbf{P} \mathbf{U}$ and its thin SVD $ \mathbf{P}^{\frac 1 2} \mathbf{U} = \mathbf{U}_y \boldsymbol{\Sigma}_y \mathbf{V}_y^\top$. Fix parameter $\beta \in (0,1]$ such that $\xi_i(\mathbf{U}) \geq \beta\, \xi_i(\mathbf{U}_y)$ for all $i \in \{N\} $, where $\xi_i(\mathbf{A}) = s^{-1} \ell_i(\mathbf{A}) $ and $ \ell_i(\mathbf{A})$ is the leverage score of the $i$-th row of $\mathbf{A}$. For $t = 1,\dots,\nu$ and $i=1, \ldots,N$ form independent diagonal sketches $\mathbf{S}_t$ with entries $\mathrm{diag}\bigl ( \gamma_i^{(t)} / \sqrt{\eta_i} \bigr )$ from independent Bernoulli variables $ \gamma_i^{(t)}$ with probabilities $\eta_i = \min\{ 1,\, c\, \xi_i(\mathbf{U}) \}$. Let $\bar{\mathbf{S}} = \frac 1 \nu \sum_{t=1}^\nu \mathbf{S}_t $, then for any $ 0 < \epsilon < 1 $,
\begin{equation}\label{concentration}
\mathbb{P} \bigl ( \bigl \| \mathbf{U}_y^\top \bar{\mathbf{S}}^\top \bar{\mathbf{S}}\, \mathbf{U}_y - \mathbf{I} \bigr \| \geq \epsilon \bigr ) \leq 2s \exp\!\left( - \frac{3 \nu c \beta\, \epsilon^2}{6s + 2s \epsilon} \right).
\end{equation}
Equivalently, to ensure $\bigl \| \mathbf{U}_y^\top \bar{\mathbf{S}}^\top \bar{\mathbf{S}}\, \mathbf{U}_y - \mathbf{I} \big\| < \epsilon $ with probability at least $1 - \delta$, it suffices that
$$
c \geq \; \frac{6s + 2s \epsilon}{3 \nu \beta\, \epsilon^2}\; \log\!\frac{2s}{\delta}.
$$
Moreover, for $ \bar{\mathbf{Y}}(\mathbf{P}) = \frac 1 \nu \sum_{t=1}^\nu \mathbf{U}^\top \mathbf{P}^{\frac 1 2} \mathbf{S}_t^\top \mathbf{S}_t \mathbf{P}^{\frac 1 2} \mathbf{U}$, we have $\bar{\mathbf{Y}}(\mathbf{P}) \succ 0$ if $ \bigl \| \mathbf{U}_y^\top \bar{\mathbf{S}}^\top \bar{\mathbf{S}}\, \mathbf{U}_y - \mathbf{I} \bigr \| < 1$.
\end{theorem}

\begin{proof}
Deferred to \ref{Proofthm5.2}.
\end{proof}

This theorem affirms that the probability of the sketching error exceeding some small threshold improves linearly with more samples from either of $c$ or $\nu$, and on the other hand it degrades linearly with larger $s$ and a smaller parameter-leverage mismatch $\beta$. As we have no control on $\beta$ with parameter-oblivious sampling, the following corollary suggests that in choosing $c \nu$ big enough offsets $\beta < 1$ while making invertibility overwhelmingly likely.  

\begin{corollary}
Under the assumptions of Theorem \ref{concentave},
\begin{equation}\label{col2}
\mathbb{P} \bigl (\bar{\mathbf{Y}}(\mathbf{P}) \not\succ 0 \bigr ) \leq 2s \exp\!\left( - \frac{3 \nu c \beta}{8s} \right).
\end{equation}
In particular, $\mathbb{P}\bigl ( \bar{\mathbf{Y}}(\mathbf{P}) \succ 0 \bigr ) \geq 1 - \delta^\star $ whenever
\begin{equation}\label{col2b}
c \;\geq \frac{8s}{3 \nu \beta}\, \log\!\frac{2s}{\delta^\star} \,.
\end{equation}
\end{corollary}

\begin{proof}
The probability $\delta^\star$ can be computed from \eqref{concentration} by setting $\epsilon = 1$.
\end{proof}

In deciding an appropriate pair of values for $c$ and $\nu$, particularly under resource constraints, we may consider reducing the per-sketch sampling budget to take advantage of the rank accumulation in the average. For reasons that will become apparent in the next section where we discuss variance reduction, we choose to keep the per-sketch sample budget $c$ around $\mathcal{O}(s \log s)$ irrespectively of $\beta$ and use the ensemble size $\nu$ to increase the overall sample size. The following proposition provides invertibility guarantees for the average of positive semidefinite (PSD) sketches under the $c < \frac{8s}{3 \nu \beta} \log \bigl ( \frac{2s}{\delta^\star} \bigr )$ regime. 

\begin{proposition}
\label{smallc}
Let $\hat{\mathbf{Y}}_1(\mathbf{P}), \ldots, \hat{\mathbf{Y}}_\nu(\mathbf{P})$ be iid PSD sketches based on \eqref{invsketch} with $s < c < \frac{8s}{3 \nu \beta} \log \bigl ( \frac{2s}{\delta^\star} \bigr )$,  and $\mathrm{rank}\bigl (\hat{\mathbf{Y}}_t(\mathbf{P}) \bigr ) \leq r$ for all $t$. Then for $\bar{\mathbf{Y}}(\mathbf{P}) = \frac 1 \nu \sum_{t=1}^{\nu} \hat{\mathbf{Y}}_t(\mathbf{P})$ we have
$$
\mathrm{rank}\bigl (\bar{\mathbf{Y}}(\mathbf{P})\bigr ) \leq \min \bigl \{s, \nu r \bigr \}.
$$
If $\nu r < s$, then $\bar{\mathbf{Y}}(\mathbf{P})$ is singular with probability 1. Conversely, $\bar{\mathbf{Y}}(\mathbf{P}) \succ 0$ if the union of sampled indices across the $\nu$ sketches contains some fixed spanning set $J_\star \subset \{N\}$ of dimension $s$, in which case \begin{equation}
\mathbb{P} \bigl ( \bar{\mathbf{Y}}(\mathbf{P}) \succ 0) \geq \prod_{i\in J_\star}\bigl ( \,1-(1-\eta_i)^{\nu}\,\bigr )
\;\geq \;
1 - \sum_{i\in J_\star}(1-\eta_i)^{\nu}.
\end{equation}
\end{proposition}

\begin{proof}
Deferred to \ref{Proofprop521}.  
\end{proof}

We conclude our analysis on $\bar{\mathbf{Y}}(\mathbf{P})$ by examining its bias and variance. From \eqref{sdiag}, since $\boldsymbol{\eta} \succ 0$
$$
\mathbb{E}[\mathbf{S}^\top \mathbf{S}] = \mathbb{E}\bigl [ \mathrm{diag} (\boldsymbol{\gamma} \circ \boldsymbol{\gamma} \circ \boldsymbol{\eta}^{-1}) \bigr ] = \mathbb{E} \bigl [ \mathrm{diag}(\boldsymbol{\gamma} \circ \boldsymbol{\eta}^{-1}) \bigr ] = \mathrm{diag}\bigl ( \mathbb{E}[\boldsymbol{\gamma}] \circ \boldsymbol{\eta}^{-1} \bigr ) = \mathbf{I},
$$
where  $\boldsymbol{\eta}^{-1}$is the vector of the reciprocals of the non-zero Bernoulli probabilities. In effect, 
\begin{equation}
\mathbb{E}[\bar{\mathbf{Y}}(\mathbf{P})] = \mathbb{E}\Bigl [ \frac 1 \nu \sum_{t=1}^{\nu} \hat{\mathbf{Y}}_t(\mathbf{P}) \Bigr ] = \frac 1 \nu \sum_{t=1}^\nu \mathbf{U}^\top \mathbf{P}^{\frac 1 2} \mathbb{E} \bigl [ \mathbf{S}_t^\top \mathbf{S}_t \bigr ] \mathbf{P}^{\frac 1 2} \mathbf{U} = \mathbf{Y}(\mathbf{P}),   
\end{equation}
and thus $\bar{\mathbf{Y}}(\mathbf{P})$ is unbiased. For the total variance of the sketch, we are essentially looking for the sum of the element-wise variances as
\begin{equation}
\mathbb{V}(\bar{\mathbf{Y}}(\mathbf{P})) \doteq \sum_{h,k=1}^{s} \mathbb{V} (\bar{y}_{hk}(\mathbf{P})) = \mathbb{E} \Bigl [ \bigl \|\bar{\mathbf{Y}}(\mathbf{P}) - \mathbf{Y}(\mathbf{P})  \bigr \|_\mathrm{F}^2 \Bigr ], 
\end{equation}
but before deriving its expression we state the following remark that exploits the symmetry of the sketch.

\begin{remark}\label{varcovar}
Let $1 \leq h\neq k \leq s$ be two distinct indices. The variance of the off-diagonal entry $\bar{y}_{hk}(\mathbf{P})$ of the  multi-sketch average estimator is equal to the covariance of the diagonal entries $\bar{y}_{hh}(\mathbf{P})$ and $\bar{y}_{kk}(\mathbf{P})$,
\begin{equation}\label{lemma2}
\mathbb{V}\bigl (\bar{y}_{hk}(\mathbf{P}) \bigr ) = \mathbb{C} \bigl (\bar{y}_{hh}(\mathbf{P}),\bar{y}_{kk}(\mathbf{P}) \bigr ). 
\end{equation}
\end{remark}

\begin{proof}
Deferred to \ref{Proofrem53}.    
\end{proof}

\begin{lemma}\label{VarYbarP}
Under the setup of theorem \ref{concentave}, the Frobenius-variance of the average sketch is
\begin{equation}
\mathbb{V}(\bar{\mathbf{Y}}(\mathbf{P})) = \frac 1 \nu \sum_{i=1}^N \Bigl ( \frac{1}{\eta_i} - 1 \Bigr ) p_{ii}^2 \ell_i(\mathbf{U})^2.     
\end{equation}
If, additionally, $\eta_i = \tfrac{c}{s} \ell_i(\mathbf{U}) $ with $\eta_i < \frac 1 2$ for all $i \in \{N\}$, then
\begin{equation}\label{Frob_conc}
\mathbb{V}(\bar{\mathbf{Y}}(\mathbf{P})) \leq  \frac 1 \nu \sum_{i=1}^N p_{ii} \Bigl ( \frac{s}{c} - \ell_i(\mathbf{U}) \Bigr ) \ell_i(\mathbf{U})
,
\quad \text{and} \quad
\mathbb{P}\!\left( \| \bar{\mathbf{Y}}(\mathbf{P}) - \mathbf{Y}(\mathbf{P}) \|_\mathrm{F} \geq \epsilon \right) \;\leq \; \frac{s}{\nu c \epsilon^2} \sum_{i=1}^N p_{ii}^2 \ell_i(\mathbf{U}).
\end{equation}
\end{lemma}

\begin{proof}
Deferred to \ref{Prooflem54}.    
\end{proof}

From \eqref{Frob_conc} notice that the tail concentration of the Frobenius norm variance effectively relies on the alignment between the squared parameters vector and the leverage scores of $\mathbf{U}$. For heterogeneous parameter fields whose mass is concentrated in regions of the domain that coincide topologically with the largest leverage scores then the variance and its tail probability grow. Before addressing the right hand side sketch, for completeness, and to aid comparison with the asymmetric $\bar{\mathbf{Z}}(\mathbf{P})$ (c.f. theorem \ref{thm:bernstein-spectral-ZP}), we also provide the tail concentration bound for the spectral norm of the sketching error in $\bar{\mathbf{Y}}(\mathbf{P})$, even though we are only concerned in the error of its inverse as in theorem \ref{concentave}. 

\begin{theorem}\label{theoremYPerr}

Let $\bar{\mathbf{Y}}(\mathbf{P}) $ be the $ s \times s$ matrix average of $\nu$ iid sketches for the same parameters $ \mathbf{P}$ and $0 < \epsilon < 1$ then
$$
\mathbb{P} \bigl ( \bigl \|\bar{\mathbf{Y}}(\mathbf{P}) - \mathbf{Y}(\mathbf{P}) \bigr \| \geq \epsilon \bigr ) \leq 2 s \exp \Bigl (- \frac{\epsilon^2/2}{\|\mathbf{V}_{\bar{\mathbf{Y}}}\| + L \epsilon/3} \Bigr ), 
$$
where 
$$
\|\mathbf{V}_{\bar{\mathbf{Y}}}\| = \frac 1 \nu \Bigl \| \sum_{i=1}^N p_{ii}^2 \Bigl ( \frac{1}{\eta_i} - 1 \Bigr ) \ell_i(\mathbf{U}) \mathbf{u}_{(i)}^\top \mathbf{u}_{(i)} \Bigr \| \quad \text{and} \quad L = \frac{s}{c}\|\mathbf{P}\|_\infty.
$$
\end{theorem}

\begin{proof}
Deferred to \ref{Proofthm55}.
\end{proof}

\subsection{On the right hand side sketch}

We now turn our attention to the multi-sketch average estimator of the rectangular $s \times m$ matrix $\mathbf{Z}(\mathbf{P}) = \mathbf{U}^\top \mathbf{P} \mathbf{U}_{\partial}$ appearing in the parametric right hand side vector of \eqref{projY}. From the definitions \eqref{qPsplit} and \eqref{sketchave} we have a sketching-induced error 
\begin{equation}
\|\bar{\mathbf{q}}(\mathbf{P}) - \mathbf{q}(\mathbf{P}) \| \leq \| \bar{\mathbf{Z}}(\mathbf{P}) - \mathbf{Z}(\mathbf{P}) \| \, \|\boldsymbol{\Sigma}_\partial \mathbf{V}_\partial^\top  \mathbf{u}^\partial \|,
\end{equation}
for which we aim to derive a tail concentration bound for the spectral norm $\|\bar{\mathbf{Z}}(\mathbf{P}) - \mathbf{Z}(\mathbf{P})\|$ and quantify the variance of the estimator. Noting that the multi-sketch estimator is unbiased since $\mathbb{E}[\bar{\mathbf{Z}}(\mathbf{P})] = \mathbb{E} [\hat{\mathbf{Z}}(\mathbf{P})] = \mathbf{Z}(\mathbf{P})$ and $\mathbb{E}[\mathbf{M}^\top \mathbf{M}] = \mathbb{E}\bigl [ \mathrm{diag} (\boldsymbol{\tilde{\gamma}} \circ \boldsymbol{\tilde{\gamma}} \circ \boldsymbol{\tilde{\eta}}^{-1}) \bigr ] = \mathbb{E} \bigl [ \mathrm{diag}(\boldsymbol{\tilde{\gamma}} \circ \boldsymbol{\tilde{\eta}}^{-1}) \bigr ] = \mathrm{diag}\bigl ( \mathbb{E}[\boldsymbol{\tilde{\gamma}}] \circ \boldsymbol{\tilde{\eta}}^{-1} \bigr ) = \mathbf{I}$, then in terms of its variance we have  
\begin{align*}
\mathbb{V}\bigl ( \bar{\mathbf{Z}}(\mathbf{P}) \bigr ) = \mathbb{E} \bigl [\|\bar{\mathbf{Z}}(\mathbf{P}) - \mathbf{Z}(\mathbf{P}) \|_\mathrm{F}^2 \bigr ] = \sum_{h=1}^s \sum_{k=1}^{m} \mathbb{V} \bigl ( \bar{z}_{hk}(\mathbf{P}) \bigr ), 
\end{align*}
and upon using the independence of $\tilde{\boldsymbol{\gamma}}$ we can compute and bound the Frobenius variance as
\begin{align}\label{VarZP} \nonumber
\mathbb{V}\bigl ( \bar{\mathbf{Z}}(\mathbf{P}) \bigr ) & = \frac{1}{\nu^2} \sum_{t=1}^{\nu} \sum_{h=1}^s \sum_{k=1}^{m} \sum_{i=1}^N \frac{\mathbb{V}(\tilde{\gamma}_i^{(t)})}{{\tilde{\eta}}^2_i} p_{ii}^2 u_{ih}^2 {u_\partial}_{ik}^2 = \frac{1}{\nu^2} \sum_{t=1}^{\nu} \sum_{i=1}^N \frac{\mathbb{V}(\tilde{\gamma}_i^{(t)})}{{\tilde{\eta}}^2_i} p_{ii}^2 \ell_i(\mathbf{U}) \ell_i(\mathbf{U}_\partial)\\ \nonumber
& = \frac 1 \nu \sum_{i=1}^N \Bigl ( \frac{1}{{\tilde{\eta}}_i} - 1 \Bigr ) p_{ii}^2 \ell_i(\mathbf{U}) \ell_i(\mathbf{U}_\partial) \leq \frac 1 \nu \sum_{i=1}^N p_{ii}^2 \Bigl ( \frac{\sum_{j=1}^N \ell_j(\mathbf{U}) \ell_j(\mathbf{U}_\partial)} {c \, \ell_i(\mathbf{U}) \ell_i(\mathbf{U}_\partial)} - 1 \Bigr ) \ell_i(\mathbf{U}) \ell_i(\mathbf{U}_\partial)\\
& = \frac 1 \nu \sum_{i=1}^N p_{ii}^2 \Bigl ( \frac{\sum_{j=1}^N \ell_j(\mathbf{U}) \ell_j(\mathbf{U}_\partial)}{c} - \ell_i(\mathbf{U}) \ell_i(\mathbf{U}_\partial) \Bigr ), 
\end{align}
where the inequality is due to the fact that the variance is maximised when $\tilde{\eta}_i  = c \zeta_i < \frac 1 2$, $\forall i \in \{N\}$ for $\zeta_i = \frac{\ell_i(\mathbf{U}) \ell_i(\mathbf{U}_\partial)}{\sum_{j=1}^N \ell_j(\mathbf{U}) \ell_j(\mathbf{U}_\partial)}$. Similar to the variance of $\bar{\mathbf{Y}}(\mathbf{P})$, notice that the variance of $\bar{\mathbf{Z}}(\mathbf{P})$ scales to the alignment between the squared parameters vector, and those of $\boldsymbol{\ell}(\mathbf{U})$ and $\boldsymbol{\ell}(\mathbf{U}_\partial)$. Empirically, we expect this to be small as variance amplification now depends on the alignment of three vectors. Using Markov's inequality for a small error $0 < \epsilon < 1$ we have
\begin{align}
\mathbb{P} \bigl ( \|\bar{\mathbf{Z}}(\mathbf{P}) - \mathbf{Z}(\mathbf{P})\|_\mathrm{F} \geq \epsilon \bigr ) = \mathbb{P} \bigl ( \|\bar{\mathbf{Z}}(\mathbf{P}) - \mathbf{Z}(\mathbf{P})\|_\mathrm{F}^2 \geq \epsilon^2 \bigr ) = \frac{1}{\epsilon^2} \mathbb{V}\bigl ( \bar{\mathbf{Z}}(\mathbf{P}) \bigr ). 
\end{align}
Before embarking on bounding the spectral norm of the sketching error we must make a remark on the dimensions of the rectangular $s \times m$ matrix $\bar{\mathbf{Z}}(\mathbf{P})$, recalling that $m$ is the number of nodes of the mesh situated on the boundary. Whilst $s$ is assumed to be sufficiently small, e.g. $\mathcal{O}(10^2)$, $m$ can be significantly larger even though still much less than $n$. In turn, proving a concentration bound via Bernstein's inequality for rectangular matrices, will yield a probability bound proportional to $s+m$ making it highly conservative. A sharper bound can be obtained by invoking the intrinsic dimension of the matrix as we show in the following theorem.   

\begin{theorem}
\label{thm:bernstein-spectral-ZP}
Let $\mathbf{U} \in \mathds{R}^{N \times s}$ and $\mathbf{U}_\partial \in \mathds{R}^{N \times m}$ have orthonormal columns. Define $\mathbf{Z}(\mathbf{P}) = \mathbf{U}^\top \mathbf{P} \mathbf{U}_\partial$. For $t = 1,\dots,\nu$ and $i=1,\ldots,N$ form independent diagonal sketches $\mathbf{M}_t$ with entries $\mathrm{diag}\bigl ( \tilde\gamma_i^{(t)} / \sqrt{\tilde\eta_i} \bigr )$ for independent Bernoulli variables $ \tilde\gamma_i^{(t)}$ with probabilities of success $\tilde\eta_i = \min\{1,\, c\, \zeta_i \}$, $\zeta_i \propto \ell_i(\mathbf{U})\, \ell_i(\mathbf{U}_\partial)$ where $\ell_i(\mathbf{U})$ and $\ell_i(\mathbf{U}_{\partial})$ denote the leverage scores of the $i$-th rows of $\mathbf{U}$ and $\mathbf{U}_{\partial}$ respectively. Let $ \bar{\mathbf{Z}}(\mathbf{P}) = \frac 1 \nu \sum_{t=1}^\nu \mathbf{U}^\top \mathbf{P}^{\frac 1 2} \mathbf{M}_t^\top \mathbf{M}_t \mathbf{P}^{\frac 1 2} \mathbf{U}_\partial$. Then for any $ \epsilon > 0$  
\begin{equation}\label{concZP}
 \mathbb{P} \bigl ( \| \bar{\mathbf{Z}}(\mathbf{P}) - \mathbf{Z}(\mathbf{P})\| \,\geq \, \epsilon \bigr )
 \;\leq \;
 2 \frac{\mathrm{Tr}(\mathbf{V}_{\bar{\mathbf{Z}}})}{\|\mathbf{V}_{\bar{\mathbf{Z}}}\|} \,\exp\!\left(
   -\,\frac{\epsilon^2 /2}{\displaystyle
    \|\mathbf{V}_{\bar{\mathbf{Z}}}\|
     + L\,\epsilon / 3 } \right),
\end{equation}
for a symmetric matrix 
$$
\mathbf{V}_{\bar{\mathbf{Z}}} = \frac 1 \nu \begin{bmatrix}
\sum_{i=1}^N p_{ii}^2 \bigl ( \frac{1}{\tilde{\eta}_i} - 1 \bigr ) \ell_i(\mathbf{U}_\partial) \mathbf{u}_{(i)}^\top \mathbf{u}_{(i)} & 0\\
0 & \sum_{i=1}^N p_{ii}^2 \bigl ( \frac{1}{\tilde{\eta}_i} - 1 \bigr ) \ell_i(\mathbf{U}) {\mathbf{u}_\partial}_{(i)}^\top {\mathbf{u}_\partial}_{(i)}
\end{bmatrix}
$$
and
$$ 
L  = \frac{1}{\nu c} \Bigl (\sum_{i=1}^N \ell_i(\mathbf{U})\ell_i(\mathbf{U}_\partial) \Bigr ) \max_i \Bigl \{ \frac{p_{ii}}{\sqrt{\ell_i(\mathbf{U}) \, \ell_i(\mathbf{U}_\partial)}} \Bigr \} 
$$
\end{theorem}

\begin{proof}
Deferred to \ref{Proofthm56}.
\end{proof}

In the tail probability \eqref{concZP} we have reverted to using the intrinsic dimension of the variance in the dilated summands comprising the error $\bar{\mathbf{Z}}(\mathbf{P}) - \mathbf{Z}(\mathbf{P})$ instead of the ambient dimension $s+m$, aiming to exploit a potential fast decay in the eigenvalues of the PSD variance matrix $\mathbf{V}_{\bar{\mathbf{Z}}}$ whose effective rank can be significantly less than the ambient dimension \cite{Tropp2012UserFriendly}. Within the exponent, $L$ is kept small unless there is an overlap between the indices where the leverage scores $\boldsymbol{\ell}(\mathbf{U})$, $ \boldsymbol{\ell}(\mathbf{U}_\partial)$ and the largest entries of $\mathbf{P}$ are concentrated. Without any explicit information on the parameters, it is the leverage score profiles $\boldsymbol{\ell}(\mathbf{U})$ and $ \boldsymbol{\ell}(\mathbf{U}_\partial)$ that effectively control which rows dominate in the symmetric and non-symmetric products that define $\mathbf{Y}(\mathbf{P})$ and $\mathbf{Z}(\mathbf{P})$. Generally speaking, any significant alignment between the largest entries of $\mathbf{P}$ with those of $\boldsymbol{\ell}(\mathbf{U})$ and $ \boldsymbol{\ell}(\mathbf{U}_\partial)$ triggers a spike in both the sketch variances and the error probability tail bounds. Since $\boldsymbol{\ell}(\mathbf{U})$ and $ \boldsymbol{\ell}(\mathbf{U}_\partial)$ are generally not aligned, $\bar{\mathbf{Y}}(\mathbf{P})$ is more prone to this unfortunate alignment effect. Variation in the parameters impacts variance quadratically and degrades the inverse sketch error through $\sqrt{\kappa(\mathbf{Y}(\mathbf{P}))}$, whilst parameter choices with very large coefficients of variation result in ill-conditioning for both $\mathbf{Y}(\mathbf{P})$ and $\mathbf{A}(\mathbf{P})$ \cite{VavasisWildCoeff}. Overall, we should expect good performance when the subspace dimension $s$ is kept modestly low, whilst having $\nu c$ large enough to compensate for suboptimal sampling when $\beta<1$, provided that parameters and leverage scores are not well aligned. 

\section{Sketch variance reduction}

The analysis of the previous section prompts the question whether for a given sampling budget we can compute more efficient estimators than $\bar{\mathbf{Y}}(\mathbf{P})$ and $\bar{\mathbf{Z}}(\mathbf{P})$ and in turn obtain a solution with smaller mean squared error than that of $\bar{\mathbf{v}}=\bar{\mathbf{Y}}^{-1}(\mathbf{P}) \bar{\mathbf{q}}(\mathbf{P})$ without compromising the  computational complexity. We argue that this is feasible using an approach that entails a control variates method for suppressing the variance in the multi-sketch average estimators followed by an optimisation problem that fuses two sketches into a more accurate one. 

\subsection{Control variates for sketches}

To improve the efficiency of the multi-sketch averages we appeal to the Monte Carlo scheme of control variates \cite{Lemieux2009}. Indeed, the variance of the entries of these estimators can be suppressed using an additive correction term that involves the sketching error of another random sketch -- the control variates -- whose expectation is either known or `cheap' to compute exactly. The level of reduction in variance scales to the squared correlation between the sought sketch and its control variate, and so the challenge is to make an appropriate choice of the latter without any high-dimensional computations.   Specific to our setting, this involves identifying a matrix $\mathbf{T}$ such that the entries of $\bar{\mathbf{Y}}(\mathbf{T})$ correlate strongly with the same-index entries of $\bar{\mathbf{Y}}(\mathbf{P})$ (and likewise an $\mathbf{R}$ such that the entries of $\bar{\mathbf{Z}}(\mathbf{R})$ correlate with those of $\bar{\mathbf{Z}}(\mathbf{P})$). 
To demonstrate our idea we focus explicitly on $\mathbf{Y}(\mathbf{P})$ noting that the procedure is similar for the non-symmetric $\mathbf{Z}(\mathbf{P})$ as we outline in \ref{SketchZCV}. 

Let $\mathds{L} \subset \mathds{R}^{N \times N}$ be a low-dimensional class of positive diagonal matrices then choosing $\mathbf{T} \in \mathds{L}$ to maximise the smallest absolute correlation between same-index entries of $\bar{\mathbf{Y}}(\mathbf{P})$ and $\bar{\mathbf{Y}}(\mathbf{T})$ constructed with the \emph{same} $\nu$ iid sketching matrices $\{\mathbf{S}_1,\ldots,\mathbf{S}_{\nu}\}$ as in \eqref{sdiag} mounts to minimising the Frobenius norm of the variance of their difference 
\begin{equation}\label{Topt}
\mathbf{T}^\star = \arg\min_{\mathbf{T} \in \mathds{L}} \mathbb{V} \bigl (\bar{\mathbf{Y}}(\mathbf{P}) - \bar{\mathbf{Y}}(\mathbf{T}) \bigr ).   
\end{equation}
Exploiting symmetry and the remark \ref{varcovar} the objective of the minimisation becomes 
\begin{align}\label{vardif} 
\mathbb{V} \bigl ( \bar{\mathbf{Y}}(\mathbf{P}) - \bar{\mathbf{Y}}(\mathbf{T}) \bigr ) = \frac 1 \nu \sum_{i=1}^N \Bigl ( \frac{1}{\eta_i} - 1 \Bigr ) \bigl ( p_{ii} - t_{ii} \bigr )^2 \ell_i^2(\mathbf{U}),
\end{align}
assuming $0 < \eta_i < 1$ for all $i \in \{N\}$. Excluding the trivial choice $\mathbf{T}=\mathbf{P}$, 
according to \eqref{vardif} it may suffice to pick a $\mathbf{T}$ that is close to $\mathbf{P}$ at the indices where $\eta_i^{-1} \ell_i^2(\mathbf{U})$ is large. In this context, a simple way to obtain a suitable $\mathbf{T}$ is by forming a low-resolution piecewise-constant approximation of $\mathbf{P}$ that trades-off approximation accuracy and computational effort. With each index in $\{N\}$ uniquely mapped to one of the $n_e$ elements of the mesh $\Omega_h$ through the FEM interpolation functions, consider the \emph{disjoint} index sets $\mathcal{I}_1, \ldots, \mathcal{I}_{\mu}$ topologically mapped to non-overlapping regions of the mesh denoted by $\mathcal{R}_1, \ldots, \mathcal{R}_\mu$ where
\begin{equation}
\{N\} = \bigcup_{j=1}^{\mu} \mathcal{I}_j, \quad \Omega_h = \bigcup_{j=1}^{\mu} \mathcal{R}_j, \quad \text{such that} \quad \mathcal{I}_j \mapsto \mathcal{R}_j, \quad \mathcal{R}_j \cap \mathcal{R}_{j'} = \emptyset \quad \text{for} \; j \neq j'.
\end{equation}
With $\mathcal{I}_j$ we associate an $s \times s$ matrix
$\mathbf{U}_{(\mathcal{I}_j)}^\top \mathbf{U}_{(\mathcal{I}_j)}$, where $\mathbf{U}_{(\mathcal{I}_j)}$ denotes the $\mathcal{I}_j$ rows of matrix $\mathbf{U}$, and upon imposing a region-wise constant approximation on $\mathbf{T}$ yields the low-dimensional vector of coefficients  
\begin{equation}\label{opttau}
\tau_j^\star = \frac{\sum_{i \in \mathcal{I}_j} w_{i} p_{ii}}{\sum_{i \in \mathcal{I}_j} w_{i}}, \quad \text{where} \quad w_i = \Bigl ( \frac{1}{\eta_i} - 1 \Bigr ) \ell_i^2(\mathbf{U}), \quad \text{for} \quad j=1,\ldots, \mu,
\end{equation}
from which we set $t_{ii} = \tau^\star_j, \; \forall i \in \mathcal{I}_j$.
The optimal coefficients $\boldsymbol{\tau}^\star$ always exist unless $\eta_i=1$ for all $i \in \mathcal{I}_j$. 
Having precomputed the $s \times s$ matrices $\mathbf{U}_{(\mathcal{I}_1)}, \ldots, \mathbf{U}_{(\mathcal{I}_\mu)}$ and the $\mu$ sums in the denominator in \eqref{opttau}, the expectation of the control variates matrix can be obtained as an $\mu$-dimensional product
\begin{equation}
\mathbb{E}[\bar{\mathbf{Y}}(\mathbf{T})] = \mathbf{U}_{(\mathcal{I}_j)}^\top \mathrm{diag} ( \boldsymbol{\tau}) \mathbf{U}_{(\mathcal{I}_j)},   
\end{equation}
in $\mathcal{O}(\mu s^2)$ complexity, and the special case of $\mu=1$ where $\mathcal{I}_1=\{N\}$ yields $\mathbf{T}$ to be a multiple of the identity matrix, and thus without loss of generality we can assign $\mathbf{E}[\bar{\mathbf{Y}}(\mathbf{T})] = \mathbf{I}$. To implement this technique it remains to decide how to decompose the domain $\Omega_h$ into the various disjoint regions. In domains with interior structure and a few dominant spatial features, these regions can be chosen to conform these geometric constraints. Otherwise, we may adopt a maximum entropy assumption on the parameters and form these regions based solely on topological adjacency, using for example the k-means algorithm as shown for example in figure \ref{fig1:Ptilde} \cite{GolubVanLoan2013}. Under the optimal choice of coefficients, this simple construction of control variates can achieve a reduction in the variance difference
$$
\mathbb{V}\bigl ( \bar{\mathbf{Y}}(\mathbf{P}) - \bar{\mathbf{Y}}(\mathbf{T}^\star) \bigr ) = \frac 1 \nu \biggl ( \sum_{i=1}^{N} w_i p_{ii}^2 - \sum_{j=1}^{\mu} \frac{\Bigl ( \sum_{i \in \mathcal{I}_j} w_i p_{ii} \Bigr )^2}{\sum_{i \in \mathcal{I}_j} w_i} \biggr ) 
$$
confirming that correlation peaks when $\{p_{ii}\}_{i \in \mathcal{I}_j}$ are well clustered around their region-wide means, or when outliers in the same set are matched with very small leverage scores. Because of the high-dimensional sums in the numerator of the definition of $\boldsymbol{\tau}^\star$ in \eqref{opttau}, the optimal coefficients can be approximated by standard Monte Carlo or importance sampling. As we show in \ref{SketchZCV} a similar approach can lead to a control variates matrix $\bar{\mathbf{Z}}(\mathbf{R})$ with $\mathbf{R} \in \mathds{L}$.
\begin{figure}
\centering
\includegraphics[width=0.32\linewidth, trim=60 42 60 0]{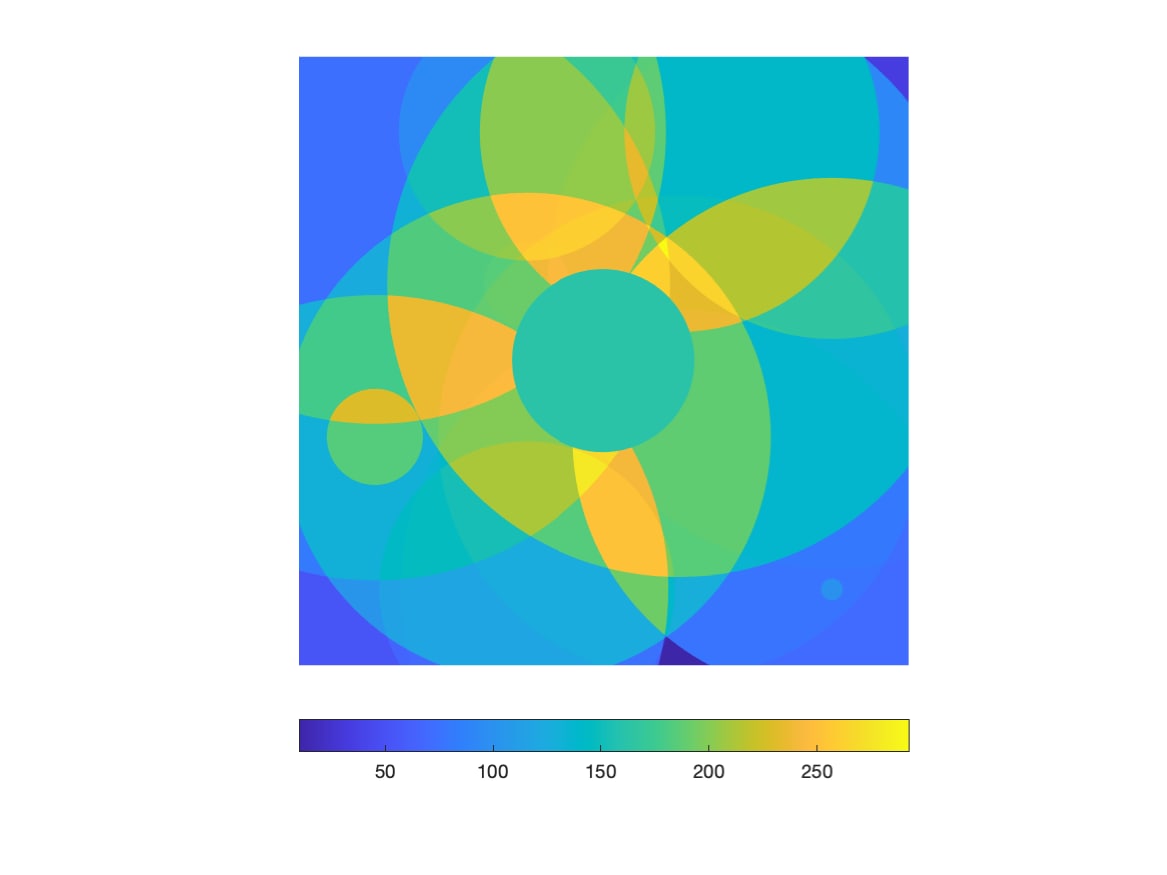}
\includegraphics[width=0.32\linewidth, trim=60 42 60 0]{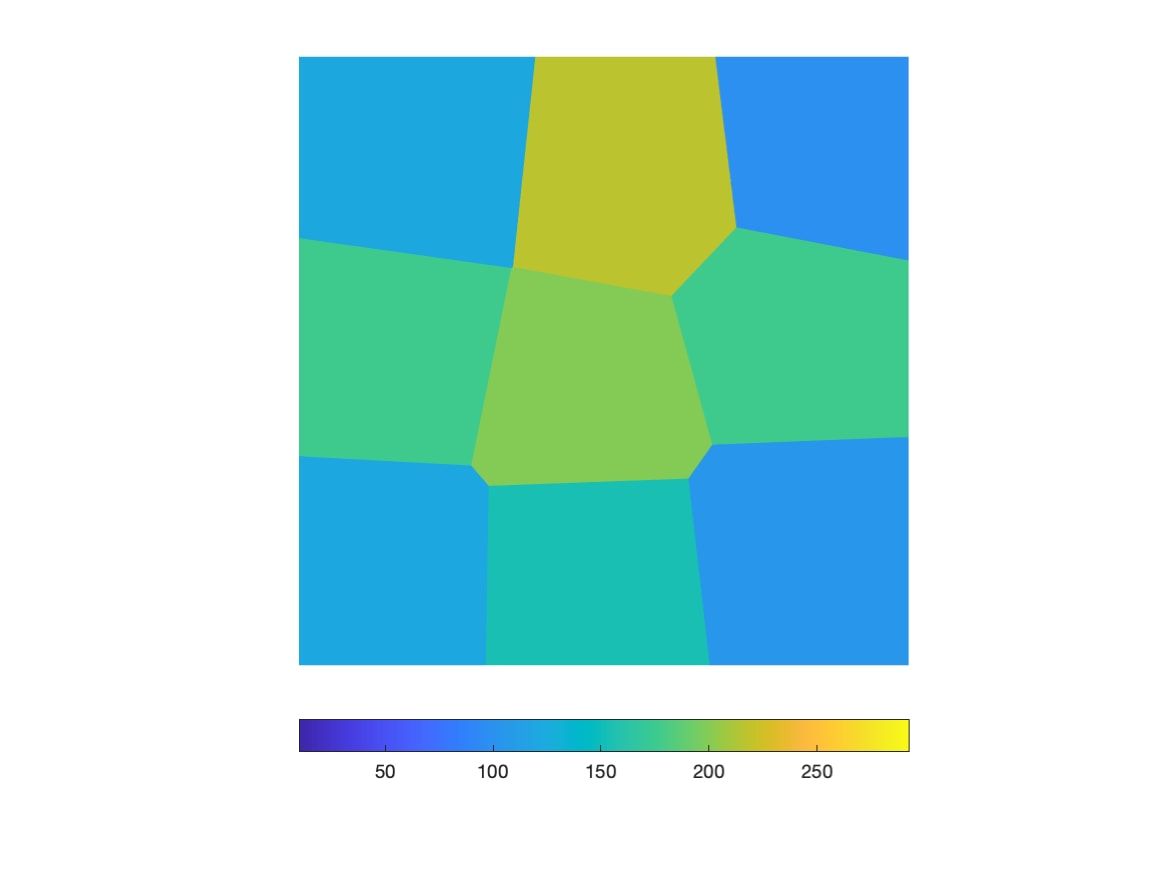}
\includegraphics[width=0.32\linewidth, trim=60 42 60 0]{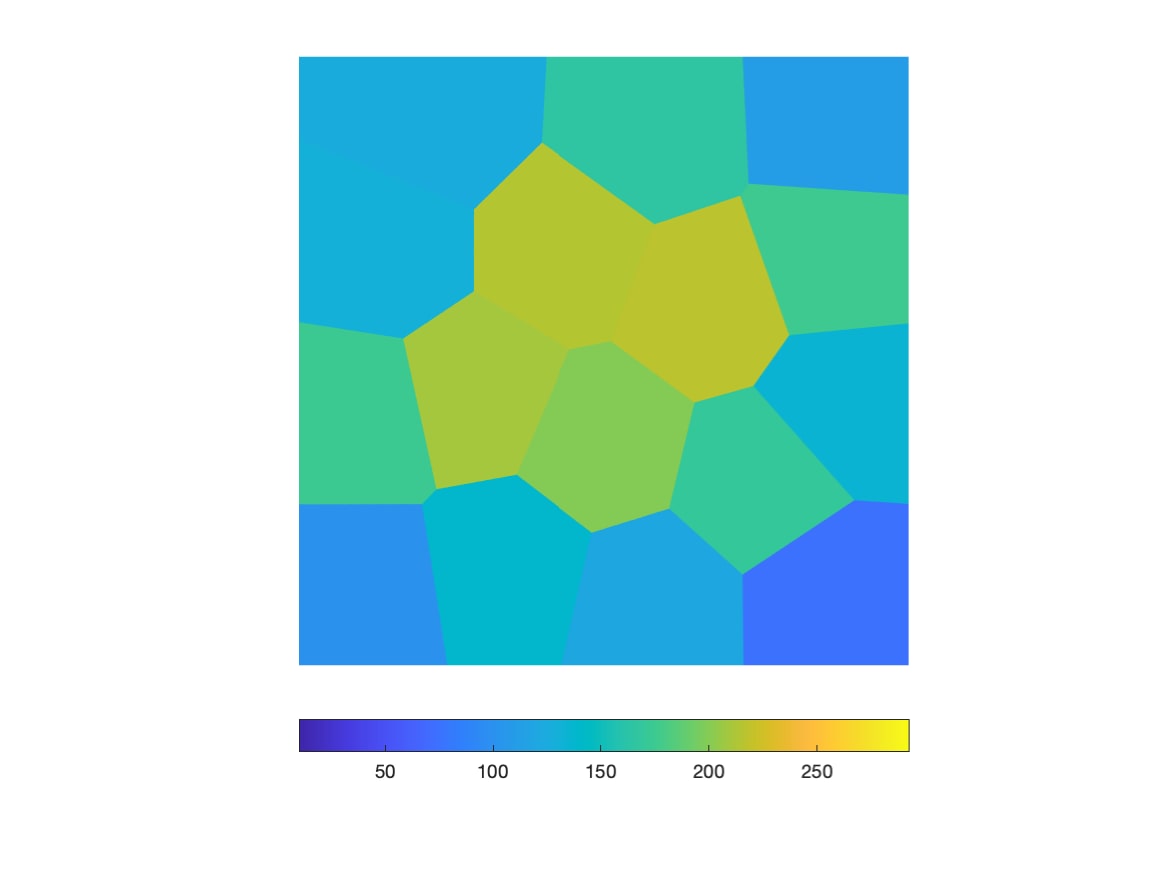}
\caption{On the left, the profile of a 2D isotropic $\mathbf{P}$ tensor field on a mesh with $n_e \approx 5 \times 10^6$ triangular elements and $N \approx 10^7$ indices. Next to it the corresponding profiles of some isotropic $\mathbf{T}$ fields with $\mu=9$ (middle) and $\mu=16$ (right) disjoint regions. The relative discrepancies $\|\mathbf{P}-\mathbf{T}\|_\mathrm{F}^2 / \|\mathbf{P}\|_\mathrm{F}^2$ are 50\% and 35\% respectively.}
\label{fig1:Ptilde}
\end{figure}
If we sketch the pairs $\bigl \{ \bar{\mathbf{Y}}(\mathbf{P}), \bar{\mathbf{Y}}(\mathbf{T}) \bigr \}$ and $\bigl \{ \bar{\mathbf{Z}}(\mathbf{P}), \bar{\mathbf{Z}}(\mathbf{R}) \bigr \}$, then the reduced variance estimators can be obtained as   
\begin{align}\label{cvformula}
\begin{cases}
\mathbf{Y}_{B}(\mathbf{P}) & = \bar{\mathbf{Y}}(\mathbf{P}) - \mathbf{B} \circ \bigl (\bar{\mathbf{Y}}(\mathbf{T}) - \mathbb{E}[\bar{\mathbf{Y}}(\mathbf{T})] \bigr ),\\ 
\mathbf{Z}_{B}(\mathbf{P}) & = \bar{\mathbf{Z}}(\mathbf{P}) - \mathbf{B}_Z \circ \bigl (\bar{\mathbf{Z}}(\mathbf{R}) - \mathbb{E}[\bar{\mathbf{Z}}(\mathbf{R})] \bigr ),
\end{cases}
\end{align}
for some weights matrices $\mathbf{B} \in \mathds{R}^{s \times s}$ and $\mathbf{B}_Z \in \mathds{R}^{s \times m}$ with optimal entries 
\begin{equation}\label{optweights}
  \begin{cases} b_{hk}^\star = \frac{\mathbb{C} \bigl ( \bar{y}_{hk} (\mathbf{P}), \bar{y}_{hk}(\mathbf{T}) \bigr )}{\mathbb{V}\bigl (\bar{y}_{hk}(\mathbf{T}) \bigr )} & \text{for} \quad 1 \leq h,k \leq s\\  {b_Z}^{\star}_{hk} = \frac{\mathbb{C} \bigl ( \bar{z}_{hk} (\mathbf{P}), \bar{z}_{hk}(\mathbf{R}) \bigr )}{\mathbb{V}\bigl (\bar{z}_{hk}(\mathbf{R}) \bigr )} & \text{for} \quad 1 \leq h \leq s, \; 1 \leq k \leq m
  \end{cases}.
\end{equation} 
The reduced variance sketches $\mathbf{Y}_{B}(\mathbf{P})$ and $\mathbf{Z}_{B}(\mathbf{P})$ typically incur a negligibly small amount of bias as the weights are computed using the same ensembles used to compute $\bar{\mathbf{Y}}(\mathbf{P})$ and $\bar{\mathbf{Z}}(\mathbf{P})$. Treating these weights as known constants leads to a conditional Frobenius norm variance  
\begin{equation}
\mathbb{V}(\mathbf{Y}_B(\mathbf{P})) = \sum_{h,k=1}^s \mathbb{V}(\bar{y}_{hk}(\mathbf{P})) (1 - \rho^2_{hk}) + \sum_{h,k=1}^s \mathbb{V}(\bar{y}_{hk}(\mathbf{T})) \mathbb{V}(\hat b_{hk}),
\end{equation}
where $\rho_{hk}$ is the correlation coefficient between $\bar{y}_{hk}(\mathbf{P})$ and $\bar{y}_{hk}(\mathbf{T})$, and $\hat b_{hk}$ is an estimator of $b^\star_{hk}$, resulting in a new Frobenius norm concentration bound 
\begin{equation}\label{CVFroconc}
\mathbb{P} \bigl ( \| \mathbf{Y}_B(\mathbf{P}) - \mathbf{Y}(\mathbf{P}) \|_\mathrm{F} \geq \epsilon \bigr ) \leq \frac{1}{\epsilon^2} \Bigl ( (1-R^2) \mathbb{V}(\bar{\mathbf{Y}}(\mathbf{P})) + \sum_{h,k=1}^s \mathbb{V}(\bar{y}_{hk}(\mathbf{T})) \mathbb{V}(\hat b_{hk}) \Bigr )
\end{equation}
with $R^2 = \sum_{h,k=1}^s \frac{\mathbb{V}(\bar{y}_{hk}(\mathbf{P}))}{\mathbb{V}(\bar{\mathbf{Y}}(\mathbf{P}))} \rho^2_{hk} \; \in [0,1]$ the weighted coefficient of determination. As $\mathbb{V}(\hat b_{hk})$ scales to $\nu^{-1}$ the second term in \eqref{CVFroconc} diminishes as the size of ensemble grows, leaving an advantage relative to the bound for the uncorrected sketch \eqref{Frob_conc} proportional to $(1-R^2)$. As far as the spectral norm of the corrected sketch's error $\|\mathbf{Y}_B(\mathbf{P})- \mathbf{Y}(\mathbf{P})\|$ is concerned, the conclusions there are more nuanced. If the sketch ensemble is restricted to small sizes, e.g. $\nu \leq 10$, the control variate weights may become too noisy and in turn yield a `corrected' sketch whose error is even worse than that of \eqref{concentration}, without achieving any reduction in variance. Also, from a linear algebra point of view, note that whilst we have $\bar{\mathbf{Y}}(\mathbf{P}) \succ 0$ with high probability, the corrected sketch may become arbitrarily ill-conditioned since from Weyl's inequality on symmetric matrices we have \cite{StewartSun}
\begin{equation}\label{Weyls}
\lambda_{\min}\bigl (\mathbf{Y}_B(\mathbf{P}) \bigr ) \geq \lambda_{\min}\bigl (\mathbf{Y}(\mathbf{P}) \bigr ) - \bigl \|\mathbf{Y}_B(\mathbf{P}) - \mathbf{Y}(\mathbf{P}) \bigr \|,
\end{equation}
with the norm term being proportional to $\|\mathbf{B}\|$ given \eqref{cvformula}, hence there is no particular lower bound for the eigenvalues of the corrected sketch. Because we utilise averages of $\nu$ sketches rather than individual sketches directly, maintaining $c \sim \mathcal{O}(s \log s)$ under tight computational constraints is highly advantageous. Scaling $\nu$ alone effectively boosts the overall sample budget, which compensates for $\beta < 1$ and enhances the control variates' ability to suppress sketch variance. Additionally, when the sampling ratio is extremely low, i.e., $c/N \ll 10^{-3}$ the computation is further accelerated by executing Bernoulli sampling via efficient geometric skipping.  

\subsection{Low variance inverse estimator}

Equipped with $\bar{\mathbf{Y}}(\mathbf{P})$, the corrected  $\mathbf{Y}_B(\mathbf{P})$, the sample-based estimators of their variances $\hat{\mathbb{V}}(\bar{\mathbf{Y}}(\mathbf{P}))$ and $\hat{\mathbb{V}}(\mathbf{Y}_B(\mathbf{P}))$, and the vector $\mathbf{q}_B(\mathbf{P}) = \bigl ( \mathbf{V} \boldsymbol{\Sigma} \bigr )^{-1} \boldsymbol{\Phi}^\top \mathbf{f} -  \mathbf{Z}_B(\mathbf{P}) \, \boldsymbol{\Sigma}_\partial \mathbf{V}_\partial^\top  \mathbf{u}^\partial$ we question whether it is possible to fuse this information in order to obtain a more accurate estimator than the multi-sketch solution in \eqref{multisketchave}, \emph{using the same sample budget} $\nu c$. For the remainder of this section we will be dealing exclusively with $\mathbf{P}$ dependent objects, we suppress the argument for clarity. Recall that due to theorem \ref{concentave}, $\bar{\mathbf{Y}}$ is SPD with probability no less than $1 - \delta^\star$, however as indicated in \eqref{Weyls} the reduced variance $\mathbf{Y}_B$ does not carry such a guarantee. The approach we advocate below requires both matrices to be SPD, and thus in the event where $\mathbf{Y}_B \nsucc 0$ a small damping term must be applied to restore positive definiteness. By virtue of the improved Frobenius norm bound established in \eqref{CVFroconc}, we now designate $\mathbf{Y}_B$ to be the \emph{forward} sketch, reflecting its higher fidelity in approximating matrix vector products involving the unknown matrix $\mathbf{Y}$ in the sense that  $\|\mathbf{Y}_B \mathbf{x} - \mathbf{Y} \mathbf{x} \|$ remains demonstrably small in probability. Conversely, we designate the uncorrected multi-sketch average estimator $\bar{\mathbf{Y}}$ to be the \emph{inverse} sketch. As guaranteed by \eqref{sketchcriterionY}, $\bar{\mathbf{Y}}^{-1}$ closely approximates the inverse of the unknown $\mathbf{Y}$ in spectral norm, thereby ensuring that the error in $\|\bar{\mathbf{Y}}^{-1} \mathbf{x} - \mathbf{Y}^{-1} \mathbf{x} \|$ is bounded with high probability. A robust way fuse the information between the forward and inverse sketches is to pose a regularised optimisation problem over the cone of $s \times s$ SPD matrices $\mathds{C}^s_{+}$, aiming to estimate directly the inverse $\mathbf{Y}^{-1}$. In effect we pose the problem
\begin{equation}\label{Coneopt}
\mathbf{H}^\star = \arg\min_{\mathbf{H} \in \mathds{C}^s_{+}}  \mathcal{J}(\mathbf{H}) \doteq \mathcal{L}(\mathbf{H} \mathbf{Y}_B, \mathbf{I}) + \mathcal{L}(\mathbf{Y}_B \mathbf{H}, \mathbf{I}) + \theta \mathcal{L}(\mathbf{H}, \bar{\mathbf{Y}}^{-1}), 
\end{equation}
where for $\mathbf{A}, \mathbf{B} \succ 0$
\begin{equation}\label{SL}
\mathcal{L}(\mathbf{A}, \mathbf{B}) \doteq \mathrm{Tr}(\mathbf{A} \mathbf{B}^{-1}) - \log \det (\mathbf{A} \mathbf{B}^{-1}) - s
\end{equation}
is the Stein's loss function $\mathcal{L}: \mathds{C}^s_{+} \times \mathds{C}^s_{+} \rightarrow [0,+\infty)$, a Bregman-type distance measure based on the logarithmic determinant divergence \cite{KANAMORI2013104}, and $\theta > 0$ is a regularisation parameter. Anchoring our search in the vicinity of the inverse sketch we seek to find a new SPD estimator whose largest eigenvalues are nearer to those of the inverse sketch, but its smallest are closer to those of the forward. The objective in \eqref{Coneopt} is well suited to our setting because Stein’s loss, through the logarithmic function, penalises heavier the underestimation of eigenvalues rather than their overestimation, which is critical in avoiding ill-conditioned estimators. Moreover, a strictly positive $\theta$ balances out the risk of compromising invertibility for the sake of statistical efficiency. Ultimately, for a choice of $\theta$ our low-variance projected FEM solution is 
\begin{equation}\label{lowvarsol}
\bar{\mathbf{w}}_\theta = (\boldsymbol{\Sigma}\mathbf{V}^\top)^{-1} \bar{\mathbf{v}}_\theta, \quad \text{where} \quad \bar{\mathbf{v}}_\theta = \mathbf{H} \, \mathbf{q}_B(\mathbf{P}), \quad  \mathbf{q}_B(\mathbf{P}) = (\mathbf{V}\boldsymbol{\Sigma})^{-1}\boldsymbol{\Phi}^\top \mathbf{f} - \mathbf{Z}_B(\mathbf{P}) \boldsymbol{\Sigma}_\partial \mathbf{V}_\partial^\top \mathbf{u}^\partial, 
\end{equation}
but before we discuss the algorithmic details of this estimator we must first address its existence and uniqueness.

\begin{theorem}\label{dualoptcomvexity}
Given SPD matrices $\bar{\mathbf{Y}}$, $\mathbf{Y}_B$, and a parameter $\theta > 0$ such that $2\mathbf{Y}_B + \theta \bar{\mathbf{Y}} \succ 0$ then $\mathcal{J}(\mathbf{H})$ in  \eqref{Coneopt} is strictly convex on the cone of $s \times s$ SPD matrices $\mathds{C}^s_{+}$. If a solution exists, then it is the unique global minimiser
\begin{equation}\label{Hoptsol}
\mathbf{H}^\star = (2+\theta)\,\bigl(2\mathbf{Y}_B + \theta \bar{\mathbf{Y}} \bigr)^{-1}.
\end{equation}
\end{theorem}

\begin{proof}
From the derivatives of the trace and log determinant functions the gradient of the objective function becomes
\begin{equation}\label{gradJ}
\nabla \mathcal{J}(\mathbf{H}) = 2 \mathbf{Y}_B + \theta \bar{\mathbf{Y}} - (2+\theta)\,\mathbf{H}^{-1},
\end{equation}
and upon differentiating again we obtain its Hessian acting on a symmetric  direction $\boldsymbol{\Delta}$ as
\begin{equation}\label{HessJ}
\nabla^2 \mathcal{J}(\mathbf{H})[\boldsymbol{\Delta}] = (2+\theta)\,\mathbf{H}^{-1}\boldsymbol{\Delta} \mathbf{H}^{-1},
\end{equation}
and hence $\mathcal{J}(\mathbf{H})$ is continuous and twice differentiable on the cone. Since for any non-zero $\boldsymbol{\Delta}$ 
\begin{align*}
\mathrm{d}^2\mathcal{J}(\mathbf{H})[\boldsymbol{\Delta},\boldsymbol{\Delta}] & = \bigl \langle \boldsymbol{\Delta}, \nabla^2 \mathcal{J}(\mathbf{H})[\boldsymbol{\Delta}] \bigr \rangle_\mathrm{F}
 = (2 + \theta) \, \mathrm{Tr}(\boldsymbol{\Delta}^\top \mathbf{H}^{-1} \boldsymbol{\Delta} \mathbf{H}^{-1} ) =(2+\theta)\,\bigl\|\,\mathbf{H}^{- \frac 1 2}\boldsymbol{\Delta} \mathbf{H}^{-\frac 1 2}\,\bigr\|_\mathrm{F}^2 > 0,
\end{align*}
is strictly positive,  $\mathcal{J}$ is strictly convex on $\mathds{C}^s_{+}$. In the quadratic form above $\langle \cdot, \cdot \rangle_\mathrm{F}$ is the Frobenius inner product. The unique minimiser can be found by setting the gradient in \eqref{gradJ} to zero.
\end{proof}

Having established the uniqueness of our new estimator we now aim to show that this always exists by proving the coercivity of the objective function at the boundary of $\mathds{C}^s_{+}$ and at infinity under the conditions of theorem \ref{dualoptcomvexity}. The following two lemmas capture this desired behaviour. 

\begin{lemma}\label{Boundaryblowup}
For any sequence of matrices $\{\mathbf{H}_k\} \subset \mathds{C}^s_{+}$ whose smallest eigenvalue approaches zero monotonically from within the cone we have
$$
\mathcal{J}(\mathbf{H}_k) \rightarrow + \infty \quad \text{as} \quad \lambda_{\min}(\mathbf{H}_k) \downarrow 0.
$$
\end{lemma}

\begin{proof}
Let the eigenvalues of $\mathbf{H}_k$ be ordered as $\lambda_{1}(\mathbf{H}_k),\ldots,\lambda_{s}(\mathbf{H}_k)>0$ with $\min_i \lambda_{i}(\mathbf{H}_k) \rightarrow 0$.
Using the properties of the trace and logarithmic determinant functions as well as the symmetry of $\bar{\mathbf{Y}}$ and $\mathbf{Y}_B$, the objective function in \eqref{Coneopt} at $\mathbf{H}_k$ expands to
$$
\mathcal{J}(\mathbf{H}_k) = \mathrm{Tr} \Bigl ( \mathbf{H}_k \bigl ( 2\mathbf{Y}_B + \theta \bar{\mathbf{Y}} \bigr ) \Bigr ) - (2+\theta) \, \log \det \mathbf{H}_k + 2 \log \det \mathbf{Y}_B - (2+\theta)s.
$$
Neglecting constant terms and noting that the trace term remains bounded from below on bounded sets, it suffices to show that for the logarithmic determinant term we have
$$
- \log \det \mathbf{H}_k = - \sum_{i=1}^s \log \lambda_{i}(\mathbf{H}_k) \rightarrow + \infty \quad \text{as} \quad \min_i \lambda_{i}(\mathbf{H}_k) \rightarrow 0,
$$
that causes $\mathcal{J}(\mathbf{H}_k) \rightarrow +\infty$ at the boundary. In essence, the objective function acts as a logarithmic barrier on the spectrum of $\mathbf{H}$ and in turn minimising sequences are bounded away from the boundary of $\mathds{C}^s_{+}$ \cite{Zenios}.
\end{proof}

To verify the coercive behaviour of the objective when the spectral norm $\|\mathbf{H}\|$ grows unbounded we pose the following lemma.  

\begin{lemma}\label{CoerciveatInf}
Under the conditions of theorem \ref{dualoptcomvexity} 
$$
\mathcal{J}(\mathbf{H}) \rightarrow + \infty \quad \text{as} \quad \|\mathbf{H}\| \rightarrow +\infty. 
$$
\end{lemma}

\begin{proof}
Define $\mathbf{K} \doteq 2 \mathbf{Y}_B + \theta \bar{\mathbf{Y}}$. With $\mathbf{K} \succ 0$ as per the conditions of the theorem, the objective function can be cast as
\begin{equation}\label{objfunct}
\mathcal{J}(\mathbf{H}) = \mathrm{Tr} \bigl ( \mathbf{H} \mathbf{K} \bigr ) - (2+\theta) \, \log \det \mathbf{H} + \text{constant}.
\end{equation}
Since $\mathbf{K} \succeq \lambda_{\min}(\mathbf{K}) \mathbf{I}$ with $\lambda_{\min}(\mathbf{K})>0$, then for any $\mathbf{H} \succ 0$ we have 
$\mathbf{H}^{\frac 1 2}\mathbf{K} \mathbf{H}^{\frac 1 2} \succeq \lambda_{\min}(\mathbf{K}) \mathbf{H}$, and then by the cyclic property of the trace we get
\begin{equation}\label{trace_linear_term}
\mathrm{Tr}(\mathbf{H}\mathbf{K}) = \mathrm{Tr}(\mathbf{H}^{\frac 1 2}\mathbf{K}\mathbf{H}^{\frac 1 2}) \geq \lambda_{\min}(\mathbf{K}) \, \mathrm{Tr}(\mathbf{H}).
\end{equation}
To address the logarithmic term in the objective we appeal to the inequality of arithmetic and geometric means on the eigenvalues of the $s \times s$ matrix $\mathbf{H}$ that asserts ${\det \mathbf{H}}^{\frac 1 s} \leq \frac 1 s \mathrm{Tr}(\mathbf{H})$. Taking the negative logarithm and rearranging yields
\begin{equation}\label{trace_log_term}
- \log \det \mathbf{H} \geq -s \log \bigl ( \frac 1 s \, \mathrm{Tr}(\mathbf{H}) \bigr ) = s \log s - s \log \mathrm{Tr}(\mathbf{H}).
\end{equation}
Introducing \eqref{trace_linear_term} and \eqref{trace_log_term} into \eqref{objfunct} we get
$$
\mathcal{J}(\mathbf{H}) \geq \lambda_{\min}(\mathbf{K}) \, \mathrm{Tr}(\mathbf{H}) - (2 + \theta) s \log \, \mathrm{Tr}(\mathbf{H}).
$$
Since $\| \mathbf{H}\| \rightarrow + \infty$ implies $\mathrm{Tr}(\mathbf{H}) \rightarrow +\infty$, then as the linear trace term is positive and dominates the logarithmic one near infinity then coercivity  $\mathcal{J}(\mathbf{H}) \rightarrow + \infty$ follows.
\end{proof}
Note that under the adopted SPD assumption on the two sketches $2\mathbf{Y}_B + \theta \bar{\mathbf{Y}} \succ 0$, a sublevel set $\bigl \{\mathbf{H} \in \mathds{C}^s_+ \, \cdot \, \mathcal{J}(\mathbf{H}) \leq \alpha \bigr \}$ at any finite $\alpha$ is bounded by lemma \ref{CoerciveatInf} as well as being bounded away from the boundary of $\mathds{C}_+^s$ by lemma \ref{Boundaryblowup}. Hence sublevel sets are contained in a compact set 
\begin{equation}
\mathds{K}_{\alpha_1,\alpha_2} \doteq \Bigl \{ \mathbf{H} \in \mathds{C}_+^s \, \cdot \, \mathrm{Tr}(\mathbf{H}) \leq \alpha_1 \, , \, \lambda_{\min}(\mathbf{H}) \geq \alpha_2 \Bigr \},    
\end{equation}
which is closed and bounded in finite dimensions for some positive $\alpha_1,\, \alpha_2$. With this set in place we can consolidate the existence proof for the new estimator in the form of the following proposition. 

\begin{proposition}\label{existence}
If $2 \mathbf{Y}_B + \theta \bar{\mathbf{Y}} \succ 0$ then $\mathcal{J}$ attains its minimum value on $\mathds{C}_+^s$, that is
\begin{equation}\label{exist}
 \exists \; \mathbf{H}^\star \in   \mathds{C}_+^s \quad \text{such that} \quad \mathcal{J}(\mathbf{H}^\star)= \inf_{\mathbf{H} \in \mathds{C}_+^s} \mathcal{J}(\mathbf{H}).
\end{equation}
\end{proposition}
\begin{proof}
Let $\{\mathbf{H}_k\} \subset \mathds{C}_+^s$ be a minimising sequence with $\mathcal{J}(\mathbf{H}_k) \downarrow \inf \mathcal{J}$. From lemmas \ref{Boundaryblowup} and \ref{CoerciveatInf}, the sequence must remain in some set $ \mathds{K}_{\alpha_1,\alpha_2}$ so it is relatively compact. Suppose we extract the $\ell$-th convergent subsequence as $\mathbf{H}_{k_\ell} \to \mathbf{H}^\dagger \in \mathds{K}_{\alpha_1,\alpha_2} \subset \mathds{C}^+_s$, then since $\mathcal{J}$ is twice continuously differentiable on the cone $\mathds{C}_+^s$ we evidently have
$$
\mathcal{J}(\mathbf{H}^\dagger) = \lim_{\ell \to \infty} \mathcal{J}(\mathbf{H}_{k_\ell}) = \inf_{\mathbf{H} \in \mathds{C}_+^s} \mathcal{J}(\mathbf{H}),
$$
and hence $\mathbf{H}^\dagger$ must be a minimiser.
\end{proof}

It now remains to discuss how to choose the regularisation parameter $\theta$ in \eqref{Coneopt}. 
Intuitively, $\theta$ should be chosen in a way that accounts for the `uncertainty' due to sketching-induced errors in $\mathbf{Y}_B$ and $\bar{\mathbf{Y}}$. Regularisation parameter selection has been central to the development of robust algorithms for ill-posed inverse problems, in particular variational problems in imaging where $\ell_2$, $\ell_1$ and total variation norms are inherent  \cite{Benning_Burger_2018}, but as we are dealing exclusively with Stein's loss distances it is important to understand the impact of these errors in our setting. Let us denote the our new estimator from \eqref{Hoptsol} as $\mathbf{H}_\theta$ to
emphasise its dependence on the regularisation parameter, and consider 
$$
\mathbf{Y}_B \doteq \mathbf{Y} + \mathbf{E}_B, \quad \bar{\mathbf{Y}} \doteq \mathbf{Y} + \bar{\mathbf{E}}, \quad \text{and} \quad \mathbf{H}_\theta \doteq \mathbf{Y}^{-1} + \boldsymbol{\Delta}_\mathbf{H}, 
$$
assuming that $\|\mathbf{E}_B\|$, $\|\bar{\mathbf{E}}\|$ and $\|\boldsymbol{\Delta}_\mathbf{H}\|$ are sufficiently small. Taking a first-order approximation of the Neumann series for the inverse function we have
$$
\mathbf{H}_\theta = \Bigl ( \mathbf{I} + \frac{1}{2+\theta}\mathbf{Y}^{-1}(2 \mathbf{E}_B + \theta \bar{\mathbf{E}}) \Bigr )^{-1} \mathbf{Y}^{-1} \approx \mathbf{Y}^{-1} - \frac{1}{2+\theta}\mathbf{Y}^{-1}(2 \mathbf{E}_B + \theta \bar{\mathbf{E}}) \mathbf{Y}^{-1},
$$
and hence the error in the inverse becomes
\begin{equation}\label{DeltaHparametric}
\boldsymbol{\Delta}_\mathbf{H} = \mathbf{H}_\theta - \mathbf{Y}^{-1} \approx  - \frac{1}{2+\theta}\mathbf{Y}^{-1}(2 \mathbf{E}_B + \theta \bar{\mathbf{E}}) \mathbf{Y}^{-1},   
\end{equation}
providing an explicit, $\theta$-parametrised expression that relates the forward and inverse sketch errors to that of the optimal solution. As $\boldsymbol{\Delta}_\mathbf{H}$ does not feature in the loss function directly we must seek a connection to the Stein loss terms so that we can relate those to $\theta$. Taking for example the term involving the forward sketch, notice that for $\mathbf{H}_\theta \mathbf{Y}_B \approx \mathbf{I}$ we have
\begin{align*}
    \mathcal{L}(\mathbf{H}_\theta \mathbf{Y}_B, \mathbf{I}) & = \mathrm{Tr} (\mathbf{H}_\theta \mathbf{Y}_B - \mathbf{I}) + s - \log\det (\mathbf{H}_\theta \mathbf{Y}_B) - s\\
    & \approx \mathrm{Tr} (\mathbf{H}_\theta \mathbf{Y}_B - \mathbf{I}) - \mathrm{Tr} \bigl ( (\mathbf{H}_\theta \mathbf{Y}_B - \mathbf{I}) - \frac 1 2 (\mathbf{H}_\theta \mathbf{Y}_B - \mathbf{I})^2 \bigr ) = \frac 1 2 \bigl \|\mathbf{H}_\theta \mathbf{Y}_B - \mathbf{I} \bigr \|_\mathrm{F}^2,
\end{align*}
where we have used the identity $\log \det (\mathbf{H}_\theta \mathbf{Y}_B) = \mathrm{Tr} \log (\mathbf{H}_\theta \mathbf{Y}_B)$, and then approximated the matrix logarithm function by its second-order Neumann series, under the assumed proximity of $\mathbf{H}_\theta \mathbf{Y}_B$ to the identity matrix. As $\mathbf{H}_\theta \mathbf{Y}_B = \mathbf{I} + \boldsymbol{\Delta}_\mathbf{H}\,  \mathbf{Y} + \mathbf{Y}^{-1}\mathbf{E}_B + \boldsymbol{\Delta}_\mathbf{H}\,  \mathbf{E}_B$, then neglecting the second-order term $\boldsymbol{\Delta}_\mathbf{H} \, \mathbf{E}_B$ and introducing into the approximate Stein loss we arrive at
$$
\mathcal{L}(\mathbf{H}_\theta \mathbf{Y}_B, \mathbf{I}) \approx \frac 1 2 \bigl  \| \boldsymbol{\Delta}_\mathbf{H}\,  \mathbf{Y} + \mathbf{Y}^{-1}\mathbf{E}_B \bigr \|_\mathrm{F}^2 = \frac 1 2 \bigl \|\mathbf{Y}^{\frac 1 2} \, \boldsymbol{\Delta}_\mathbf{H} \, \mathbf{Y}^{\frac 1 2} + \mathbf{Y}^{-\frac 1 2} \mathbf{E}_B \mathbf{Y}^{-\frac 1 2} \bigr \|_\mathrm{F}^{2}
$$
where the equality is by the cyclic property of the trace. Combining with \eqref{DeltaHparametric} yields
\begin{equation}
\mathcal{L}(\mathbf{H}_\theta \mathbf{Y}_B, \mathbf{I}) \approx \frac{1}{2(2+\theta)^2} \bigl \| (4+\theta)\mathbf{Y}^{-\frac 1 2} \mathbf{E}_B \mathbf{Y}^{-\frac 1 2} + \theta \mathbf{Y}^{-\frac 1 2} \bar{\mathbf{E}} \mathbf{Y}^{-\frac 1 2} \bigr \|_\mathrm{F}^2,  
\end{equation}
that shows explicitly that both the forward and inverse errors propagate to the loss function after transformed by a left and right multiplication with $\mathbf{Y}^{-\frac 1 2}$. The case of the Stein's loss function involving the inverse sketch follows similarly. Consequently, to find the optimal $\theta$ parameter we can consider 
\begin{equation}\label{min_theta_exp}
\theta^\star = \arg\min_{\theta>0} \mathbb{E} \bigl [ \bigl \| \mathbf{Y}^{\frac 1 2} \, \boldsymbol{\Delta}_\mathbf{H} \, \mathbf{Y}^{\frac 1 2} \bigr  \|_{\mathrm{F}}^{2} \bigr ],    
\end{equation}
which after importing $\boldsymbol{\Delta}_\mathbf{H}$ from \eqref{DeltaHparametric} becomes
\begin{align*}
\theta^\star = & \arg\min_{\theta>0} \frac{1}{(2+\theta)^2} \mathbb{E} \Bigl [ \bigl \| 2 \mathbf{Y}^{-\frac 1 2} \mathbf{E}_B \mathbf{Y}^{-\frac 1 2} + \theta \mathbf{Y}^{-\frac 1 2} \bar{\mathbf{E}} \mathbf{Y}^{-\frac 1 2} \bigr \|_\mathrm{F}^2 \Bigr ].
\end{align*}
Setting the derivative with respect to $\theta$ equal to zero gives the optimal value of the regularisation parameter
\begin{equation}\label{thetaopt}
\theta^\star = 2 \frac{\mathbb{E} \bigl [ \|\mathbf{Y}^{-\frac 1 2} \mathbf{E}_B\mathbf{Y}^{-\frac 1 2} \|_\mathrm{F}^2 \bigr ] - \mathbb{E} \bigl [ \langle \mathbf{Y}^{-\frac 1 2} \mathbf{E}_B \mathbf{Y}^{- \frac 1 2}, \mathbf{Y}^{-\frac 1 2} \bar{\mathbf{E}} \mathbf{Y}^{- \frac 1 2} \rangle_\mathrm{F} \bigr ] }{ \mathbb{E} \bigl [ \|\mathbf{Y}^{-\frac 1 2} \bar{\mathbf{E}}\mathbf{Y}^{-\frac 1 2} \|_\mathrm{F}^2 \bigr ] - \mathbb{E} \bigl [ \langle \mathbf{Y}^{-\frac 1 2} \mathbf{E}_B \mathbf{Y}^{- \frac 1 2}, \mathbf{Y}^{-\frac 1 2} \bar{\mathbf{E}} \mathbf{Y}^{- \frac 1 2} \rangle_\mathrm{F} \bigr ] }.    
\end{equation}
As the covariance term has to be smaller than the Frobenius variance norms for boundedness and positivity, the above can be approximated for compute savings to 
\begin{equation}\label{thetasubopt}
\theta^\star \approx 2 \frac{\mathbb{E} \bigl [ \|\mathbf{Y}^{-\frac 1 2} \mathbf{E}_B\mathbf{Y}^{-\frac 1 2} \|_\mathrm{F}^2 \bigr ]}{\mathbb{E} \bigl [ \|\mathbf{Y}^{-\frac 1 2} \bar{\mathbf{E}}\mathbf{Y}^{-\frac 1 2} \|_\mathrm{F}^2 \bigr ]},  
\end{equation}
but in either case $\mathbf{Y}$ is unknown and thus we must invoke a suitable surrogate for it, either through independent sampling or by appealing to its control variates counterpart. On the other hand, if the matrix is expected to be well-conditioned, e.g. when the parameters do not vary wildly, we may further approximate $\mathbf{Y} \approx \mathbf{I}$ and thus replace  \eqref{thetasubopt} with
$$
\hat{\theta} = 2 \frac{\mathbb{E} \bigl [ \|\mathbf{E}_B\|_\mathrm{F}^2 \bigr ]}{\mathbb{E} \bigl [ \|\bar{\mathbf{E}} \|_\mathrm{F}^2 \bigr ]} = 2 \frac{\mathbb{V}(\mathbf{E}_B)}{\mathbb{V}(\bar{\mathbf{E}})} = 2 \frac{\mathbb{V}(\mathbf{Y}_B)}{\mathbb{V}(\bar{\mathbf{Y}})}
$$
noting that the Frobenius variances are readily available without any additional effort from the control variates process.

Before we outline our algorithm there is an important observation to be made regarding the sampling process. Exploiting that the per-sketch budget $c$ is deliberately kept small in comparison to $N$ as we pass on the control to increase the overall sample to $\nu$, when $c/N \ll 10^{-3}$ is very small we opt to implement the Bernoulli sampling indirectly, effectively reducing the per-sketch complexity from $\mathcal{O}(N)$ to $\mathcal{O}(c)$ and thereby achieving a radical speed up of the row sampling process. As outlined in algorithm \ref{alg:fast_thinning}, we can regard the independent sampling of about $c$ of the $N$ rows of a tall matrix as a `sequence' of independent trials under a homogeneous Bernoulli process with probability  $\eta_{\max}=\min_i \eta_i < 1$. If that were to be the case, then the small number of rows can be sampled by exploiting the fact that the interval between successively selected indices in this sequence follows a geometric distribution with parameter $\eta_{\max}$ \cite{bertsekas2008introduction}. The saving comes in that after the first row is selected, we then sample by performing random jumps skipping over row intervals with a Geometric distribution of parameter $\eta_{\max}$ rather than Bernoulli sampling of individual rows. As the Bernoulli probabilities $\boldsymbol{\eta}$ and $\boldsymbol{\tilde{\eta}}$ are not identical however, this geometric skipping is then followed by a rejection sampling scheme (also referred to as Bernoulli thinning) to compensate for any $\eta_i \neq \eta_{\max}$. Once a row is selected by the geometric sampling, then this is only accepted with probability $\eta_i/\eta_{\max}$. Using basic probability rules, if $\mathbb{P}(\gamma_i^{\max}=1) =\eta_{\max}$, and $\mathbb{P}(\gamma_i^{\mathrm{rej}}=1) = \eta_i / \eta_{\max}$ the $i$-th row is selected iff it gets selected in the geometric skipping and does not get rejected in the thinning process, which turns to be equivalent to the intended Bernoulli sampling since $\mathbb{P}(\gamma_i=1) = \mathbb{P}(\gamma_i^{\max}=1) \, \mathbb{P}(\gamma_i^{\mathrm{rej}}=1 \, | \, \gamma_i^{\max} = 1) = \eta_i$. This process is outlined in some detail in the description of the algorithm \ref{alg:fast_thinning} next as part of the main algorithm for our method presented in algorithm \ref{alg:main}.

\begin{algorithm}[H] 
\caption{Low-variance parametric sketched solver}
\label{alg:main}
\begin{algorithmic}[1] 
\STATE \textbf{Input:} Parameters $\mathbf{P}$, probabilities $\boldsymbol{\eta}$ and $\boldsymbol{\tilde{\eta}}$, row index subsets $\{\mathcal{I}_1,\ldots, \mathcal{I}_\mu\}$, per-sketch row index budget $c$, sketches budget $\nu$, and invertible matrix $\boldsymbol{\Sigma}\mathbf{V}^\top \in \mathds{R}^{s \times s}$ 
\STATE \textbf{Require access} to individual rows of $\mathbf{U} \in \mathds{R}^{N \times s}$ and $\mathbf{U}_\partial \in \mathds{R}^{m \times s}$
\STATE \textbf{Output:} Low-variance sketched projected solution $\bar{\mathbf{w}}_\theta$

\STATE Estimate the coefficients of the low-dimensional $\mathbf{T}, \mathbf{R} \in \mathds{L}$ 
\STATE Evaluate control variates matrices $\mathbf{Y}(\mathbf{T}) \in \mathds{R}^{s \times s}$ and $\mathbf{Z}(\mathbf{R}) \in \mathds{R}^{m \times s}$
\STATE Run sampling \textbf{Algorithm \ref{alg:fast_thinning}} with inputs ($c$, $\boldsymbol{\eta}$, $\mathbf{P}$, $\mathbf{T}$, $\nu$), and ($c$, $\boldsymbol{\tilde{\eta}}$, $\mathbf{P}$, $\mathbf{R}$, $\nu$) 
    \STATE Compute multi-sketch averages  $\bar{\mathbf{Y}}(\mathbf{P})$, $ \bar{\mathbf{Y}}(\mathbf{T})$, $ \bar{\mathbf{Z}}(\mathbf{P})$ and $\bar{\mathbf{Z}}(\mathbf{R})$  
\STATE Estimate element-wise covariances $\mathbb{C} \bigl (\bar{\mathbf{Y}}(\mathbf{P}), \bar{\mathbf{Y}}(\mathbf{T}) \bigr )$ and  $\mathbb{C} \bigl (\bar{\mathbf{Z}}(\mathbf{P}), \bar{\mathbf{Z}}(\mathbf{R}) \bigr )$
\STATE Estimate weights $\hat{\mathbf{B}}$ and $\hat{\mathbf{B}}_Z$ and parameter $\theta$
\STATE Evaluate the low-variance estimators $\mathbf{Y}_B(\mathbf{P})$, $\mathbf{Z}_B(\mathbf{P})$ and $\bar{\mathbf{q}}_B(\mathbf{P})$
\STATE Compute $\mathbf{H} = (2+\theta)\bigl (2\mathbf{Y}_B(\mathbf{P})+ \theta \bar{\mathbf{Y}}(\mathbf{P}) \bigr )^{-1}$ 
\STATE Compute $\bar{\mathbf{v}}_\theta = \mathbf{H} \, \mathbf{q}_B(\mathbf{P})$
\STATE \textbf{Return:} $\bar{\mathbf{w}}_\theta = (\boldsymbol{\Sigma} \mathbf{V}^\top)^{-1} \bar{\mathbf{v}}_\theta$ 
\end{algorithmic}
\end{algorithm}

\begin{algorithm}
\caption{Bernoulli sampling via Geometric Skipping \& Bernoulli Thinning}
\label{alg:fast_thinning}
\begin{algorithmic}[1]
\STATE \textbf{Input:} Parameters $\mathbf{P}$, $\mathbf{T}$, probabilities $\boldsymbol{\eta}$, and iterations $\nu$.
\STATE \textbf{Require access} to individual rows of $\mathbf{U} \in \mathds{R}^{N \times s}$ 
\STATE \textbf{Output:} Sampled $\{\hat{\mathbf{Y}}_t(\mathbf{P})\}_{t=1}^{\nu}$, and $\{\hat{\mathbf{Y}}_t(\mathbf{T})\}_{t=1}^{\nu}$ 
\STATE \textbf{Initialize:}  $\eta_{\max} = \max(\boldsymbol{\eta})$.
\FOR{$t = 1$ \textbf{to} $\nu$}
\STATE $k \gets 0$, $\mathcal{I} \gets \emptyset$
\WHILE{$k < N$}
    \STATE \textbf{Geometric sampling with probability $\eta_{\max}$}:
    \STATE $k \gets k + \lceil \log(u) \, / \log(1 - \eta_{\max}) \rceil$ with $u \sim \mathcal{U}(0,1)$.
    \STATE \textbf{Rejection sampling}: 
    \STATE $\mathcal{I} \gets \mathcal{I} \cup \{k\}$ \textbf{if} $v < \eta_k / \eta_{\max}$ where $v \sim \mathcal{U}(0,1)$.
    \ENDWHILE
    \STATE  \textbf{Extract sampled rows} $\mathbf{U}_{(\mathcal{I})}$ and diagonals $\mathbf{P}_{(\mathcal{I}, \mathcal{I})}$ and $\mathbf{T}_{(\mathcal{I}, \mathcal{I})}$ 
    \STATE \textbf{Return:} $\hat{\mathbf{Y}}_t(\mathbf{P}) \gets \mathbf{U}_{(\mathcal{I})}^\top \mathbf{P}_{(\mathcal{I}, \mathcal{I})} \mathbf{U}_{(\mathcal{I})}$, \; $\hat{\mathbf{Y}}_t(\mathbf{T}) \gets \mathbf{U}_{(\mathcal{I})}^\top \mathbf{T}_{(\mathcal{I}, \mathcal{I})} \mathbf{U}_{(\mathcal{I})}$ .
\ENDFOR
\end{algorithmic}
\end{algorithm}

\section{Numerical experiments}

To test the performance of our main algorithm \ref{alg:main} we consider the inhomogeneous Poisson equation with variable coefficients as in PDE \eqref{pde} in two and three physical dimensions. Aside the numerical evidence on the practical performance of our approach, ultimately we aim to verify how much of the variance inflation incurred by not using the optimal, parameter-dependent, leverage scores can be mitigated by the synergistic action of the control variates scheme and the `forward-inverse' sketch fusion. Since our parameter-oblivious sampling is effectively optimal for constant parameters (spatial homogeneity), it is key to understand how does the variation in the profile of the parameters affects the performance of our method and how do we set the sample budget $(\nu, c)$ and control variates dimension $\mu$ to achieve the best results possible. To be systematic, we will describe this variability in the parameters by means of the coefficient of variation (CoV) 
\begin{equation}\label{coefvar}
c_V(\mathbf{P}) = \sqrt{N \, \frac{\|\mathrm{diag}(\mathbf{P})\|^2}{\|\mathrm{diag}(\mathbf{P})\|_1^2} - 1},
\end{equation}
that takes values in $[0,\sqrt{N-1}]$. 

\subsection{Two-dimensional example}
In the first set of numerical experiments we consider a square domain geometry $\Omega \in [-1, 1] \times [-1, 1]$ embedded on the $xy$ plane, with fixed forcing and boundary conditions as
$$
f(x,y) = \begin{cases}
 1 & \text{if} \quad 6(x^2+y^2)^2 + x^3 - 3xy^2 < 0\\
 0 & \text{otherwise}
\end{cases}, \quad u^\partial(x,y) = \frac 1 4 \sin \bigl ( \frac{\pi}{2} (x+y)\bigr ).
$$
The domain is discretised to a mesh with $n_n = 11212033$ nodes connected into $n_e = 22413312$ linear triangular elements. Among the $n_n$ nodes, $m=10752$ are on the boundary, yielding $n=11201281$ degrees of freedom and a large dimension $N=2 n_e \approx 45$ million. To generate the parameter profiles for the various choices of $\mathbf{P}$ diagonals we restrict to isotropic tensor fields. In each instance, we simulate between 36 and 81 discs of arbitrarily chosen radii and scatter them randomly within the closure of the domain, see for example the image on the left of figure  \ref{fig:imageslowCov}. Each of these disks is assigned a positive value from the interval $[10^{-2}, 10^{2}]$ uniformly at random, and $\mathbf{P}$ is eventually formed by adding the values of the various overlapping disks at any given element of the mesh. Using a POD-obtained subspace projection we set the small dimension to $s=200$, and then fix the per-sketch sample budget to $c=\lceil 5 s \log s \rceil \approx 5300 $, effectively corresponding to a $\beta \geq 0.2$ in theorem \ref{concentave}. Running our solver in algorithm \ref{alg:main} on sets of 100 independently chosen $\mathbf{P}$, with $(\mu, \nu) \in \{1, 16\} \times \{10,100\}$ and $\nu \in \{10,100\}$ for parameters have low CoV and $(\mu, \nu) \in \{1, 16\} \times \{50,500\}$ for those with higher CoV yields eight sets of structured and complementary experiments whose results are graphically illustrated in figure \ref{fig:2Dresults}. In this figure the panels are arranged in increasing $c_V(\mathbf{P})$ from top row to bottom, and in increasing control variates regions $\mu$ from left to right. Columns 1 to 5 depict the statistical variation of the key error quantities $\|\mathbf{Y}^{-1} - \bar{\mathbf{Y}}^{-1}\| / \|\mathbf{Y}^{-1}\|$, $\|\mathbf{Y} - \bar{\mathbf{Y}}\|_\mathrm{F} / \|\mathbf{Y}\|_\mathrm{F}$, $\|\mathbf{Y} - \mathbf{Y}_B \|_\mathrm{F} / \|\mathbf{Y}\|_\mathrm{F}$, $\|\bar{\mathbf{w}}_\theta - \mathbf{w}\| / \|\mathbf{w}\|$, and $\|\bar{\mathbf{w}} - \mathbf{w}\| / \|\mathbf{w}\|$ over the 100 runs, while column 6 holds $c_V(\mathbf{P})$ for the same set of experiments. We thus expect that on average, i.e., where peak of violin plot is, the errors in the 3rd column are less than those in the 2nd because of the control variates correction, and those in the 4th column showing the error of our improved estimator to be less than the classical multi-sketch estimator at the 5th column, for all choices of $\nu$.   

In more detail, for the top panels in figure \ref{fig:2Dresults} where the CoV is kept moderately low, note that the relative errors of $\bar{\mathbf{Y}}^{-1}(\mathbf{P})$ in spectral norm and that of the Frobenius norm of $\bar{\mathbf{Y}}(\mathbf{P})$ appear to be on average very similar, as one would expect when $\mathbf{P}^{\frac 1 2}$ does not deviate substantially from a constant diagonal. In this situation $\boldsymbol{\xi}(\mathbf{U})$ approximates about equally well the optimal $\boldsymbol{\xi}(\mathbf{P}^{\frac 1 2} \mathbf{U})$ and $\boldsymbol{\zeta}(\mathbf{P}^{\frac 1 2} \mathbf{U},\mathbf{P}^{\frac 1 2} \mathbf{U}_\partial)$. By contrast, in the respective columns of the bottom row where $c_V(\mathbf{P})$ is larger, $\boldsymbol{\xi}(\mathbf{U})$ approximates much better $\boldsymbol{\xi}(\mathbf{P}^{\frac 1 2} \mathbf{U})$, because $\boldsymbol{\zeta}(\mathbf{P}^{\frac 1 2} \mathbf{U},\mathbf{P}^{\frac 1 2} \mathbf{U}_\partial)$ scales linearly to the parameters. This explains why for higher CoV the errors in $\bar{\mathbf{Y}}(\mathbf{P})$ are on average somewhat higher compared to those of $\bar{\mathbf{Y}}^{-1}(\mathbf{P})$. On the other hand, the errors for $\bar{\mathbf{Y}}$ over the 100 runs are significantly more concentrated than those of $\bar{\mathbf{Y}}^{-1}$ irrespectively of CoV. This behaviour arises because the direct sketching error $\bar{\mathbf{Y}} - \mathbf{Y}$ is bounded by the largest eigenvalue of the sketch, meaning its sensitivity to $\mathbf{P}$ is governed solely by variations in the largest eigenvalue of $\bar{\mathbf{Y}}(\mathbf{P})$ for each $\mathbf{P}$. Conversely, the error in the inverse $\bar{\mathbf{Y}}^{-1} - \mathbf{Y}^{-1}$ depends on the inverse of the smallest eigenvalue. This quantity is inherently larger and scales sharply as $\kappa(\mathbf{Y}(\mathbf{P}))$ grows across the parameter space as indicated by \eqref{boundfactor}. Consequently, the inverse error is significantly more sensitive to spectral outliers and exhibits greater dispersion across parameter realisations.

Comparing the errors between control variates corrected and uncorrected sketches shows that our simple scheme is always effective in reducing the variance, perhaps somewhat surprisingly even at the small ensemble sizes $\nu=10$. The reduction is more pronounced at the top row compared with the bottom, where $c_V(\mathbf{P})$ is lower, at the right column rather than the left where $\mu$ is larger and when the ensemble size $\nu$ is bigger than smaller, as this provides additional accuracy in the Monte Carlo approximation of the weights in \eqref{optweights}. Finally we compare the relative errors in the proposed low-variance estimator $\bar{\mathbf{w}}_\theta$ as in \eqref{lowvarsol} plotted in column 4 against the multi-sketch average $\bar{\mathbf{w}}$ as in \eqref{multisketchave} in column 5, subject to an important caveat. To compare them under an identical computational and sampling budget, $\bar{\mathbf{w}}$ is computed using double the ensemble size at $2 \nu$. This matches, in expectation, the total number of samples allocated to $\bar{\mathbf{w}}_\theta$, which involves $\nu$ sketches for the sought estimator and another $\nu$ sketches for its corresponding control variates matrix. Looking across the four panels of figure \ref{fig:2Dresults} we see that on either ensemble size, on average our estimator $\bar{\mathbf{w}}_\theta$ is superior in terms of the mean squared error. The improvement is considerably larger in the lower CoV regime, where as expected, the error drops with increasing $\mu$ or $\nu$. This is because a higher $\mu$ yields a higher correlation between the sketch its control variates, while a larger $\nu$ allows for more accurate estimation of the control variate weights and the regularisation parameter. In particular, after 100 independent runs for different parameters at $\nu=10$ (blue violins) the average relative errors in $\bar{\mathbf{w}}$ is 0.0084 while that of $\bar{\mathbf{w}}_\theta$ is 0.0047 for $\mu=1$ and 0.0035 for $\mu=16$. At $\nu=100$ (orange violins) the errors for $\bar{\mathbf{w}}$ average at 0.0026 whilst those for $\bar{\mathbf{w}}_\theta$ drop to 0.0014 for $\mu=1$ and 0.0011 at $\mu=16$, so overall the error is halved. Interestingly however, notice that our estimator performs only marginally better than the multi-sketch estimator in the higher CoV regime as shown in the bottom left panel where the respective average errors are much closer together at 0.0037 and 0.0040 when $\nu=50$ and then 0.0009 and 0.0012 for $\nu=500$. This indicates that the advantage of the proposed approach wears out as $\mathbf{P}$ begins to vary wildly. Our findings indicate that for CoV exceeding 2, the low-variance estimator offers no practical advantage unless a more sophisticated domain decomposition strategy is adopted for control variate assignment. Crucially however, note that in such CoV regimes, the original FEM coefficients matrix in \eqref{fempde} inherently becomes ill-conditioned. To visualise the spatial features of the sketched solution once projected to its original high-dimensional space and its discrepancy against the deterministic projected solution we plot a representative outcome in figure \ref{fig:imageslowCov} corresponding to one of the low CoV experiments. The sketching error has a distinct non-uniform profile with patterns that follow the underlying parameter field.   

\begin{figure}
\centering
\includegraphics[width=0.49\linewidth]{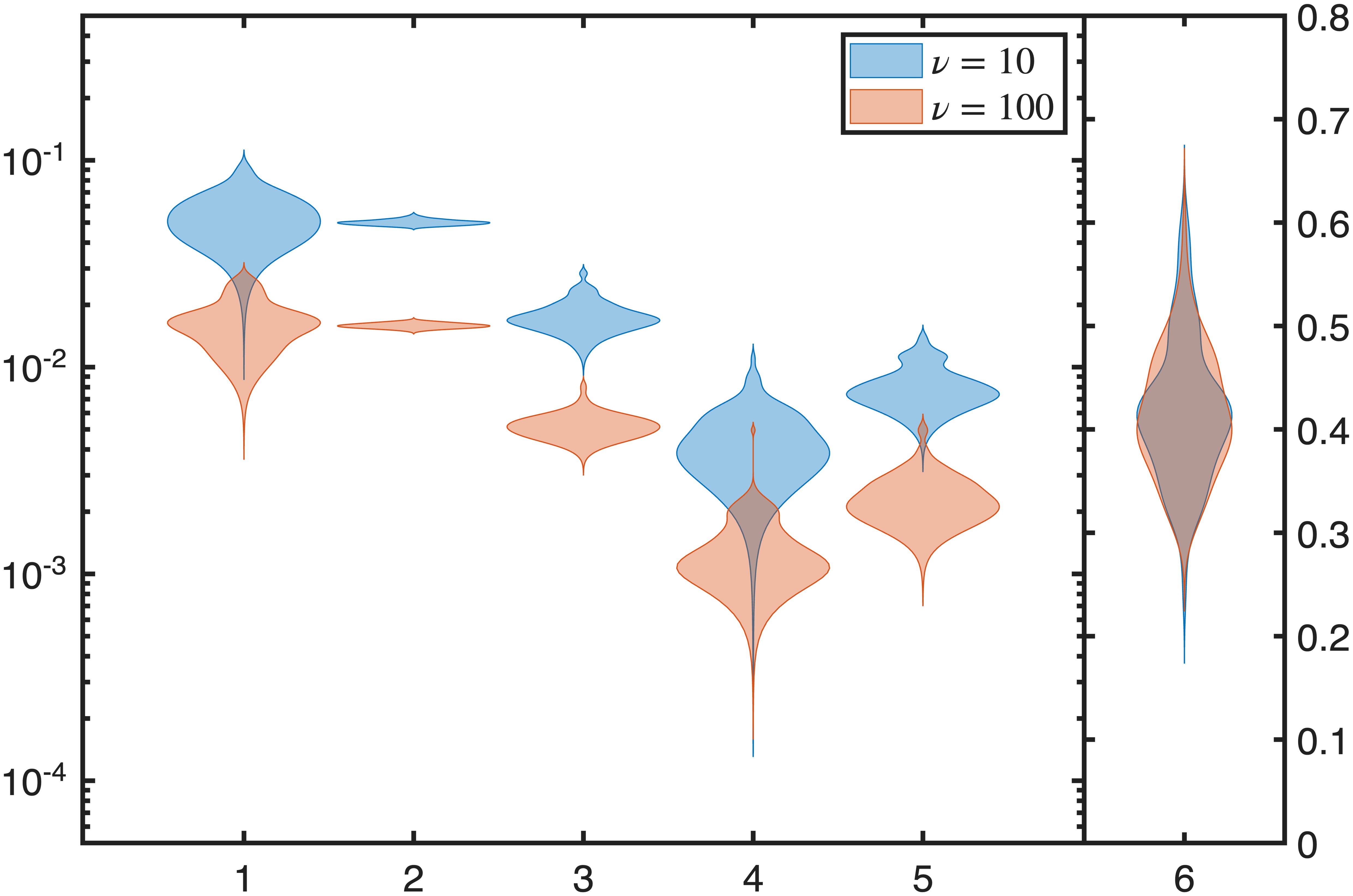}\includegraphics[width=0.49\linewidth]{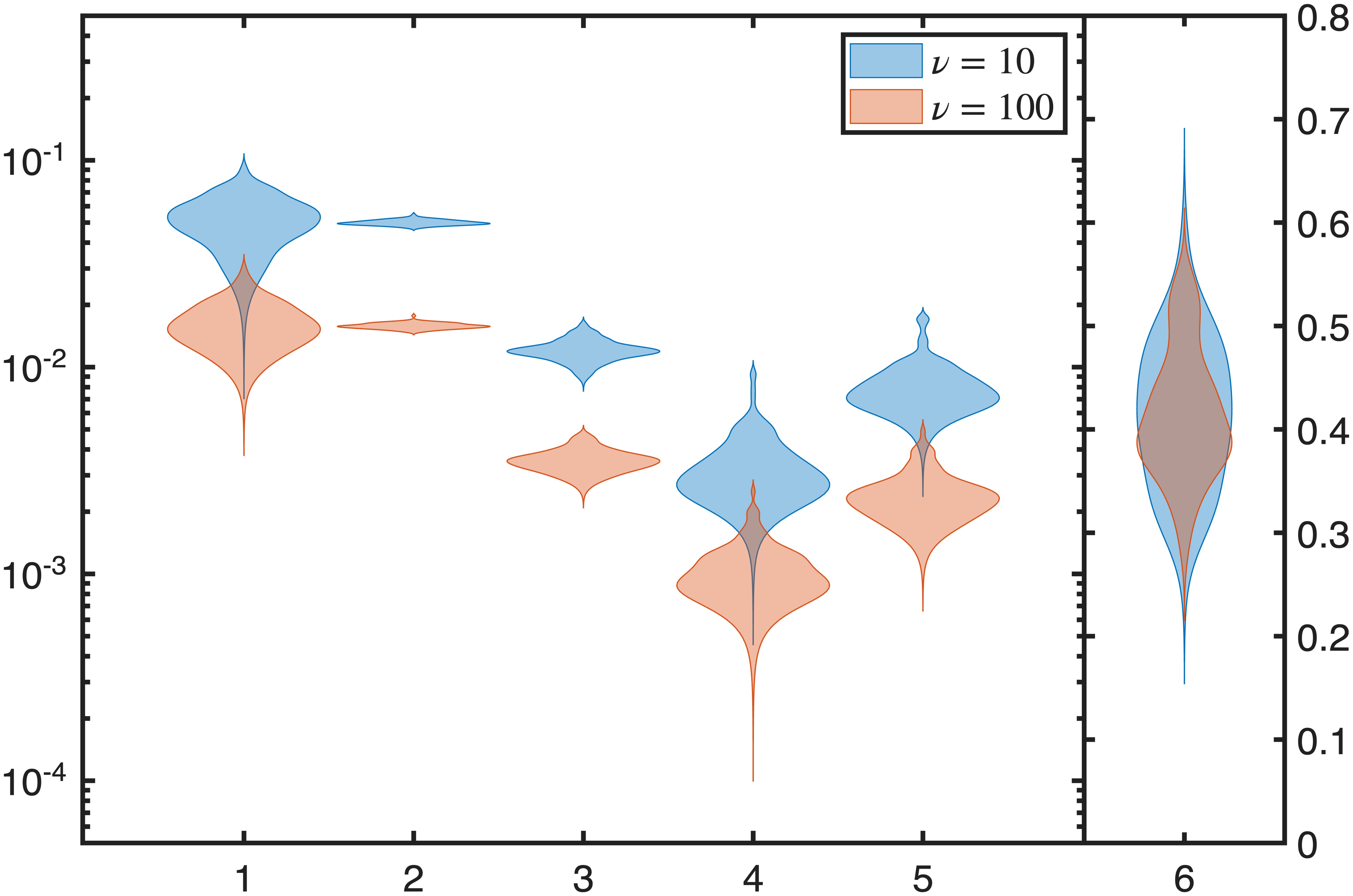}\\
\includegraphics[width=0.49\linewidth]{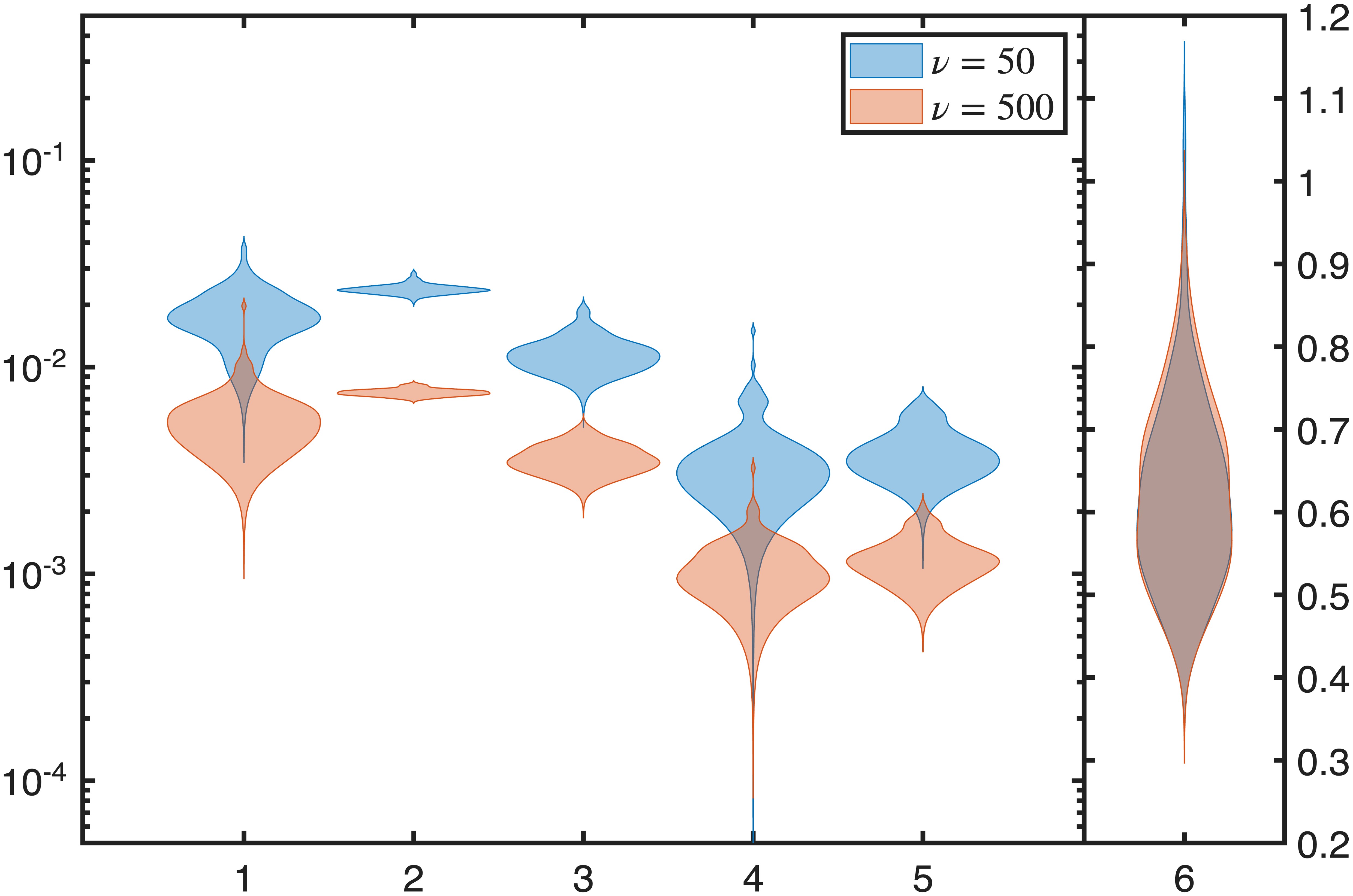} \includegraphics[width=0.49\linewidth]{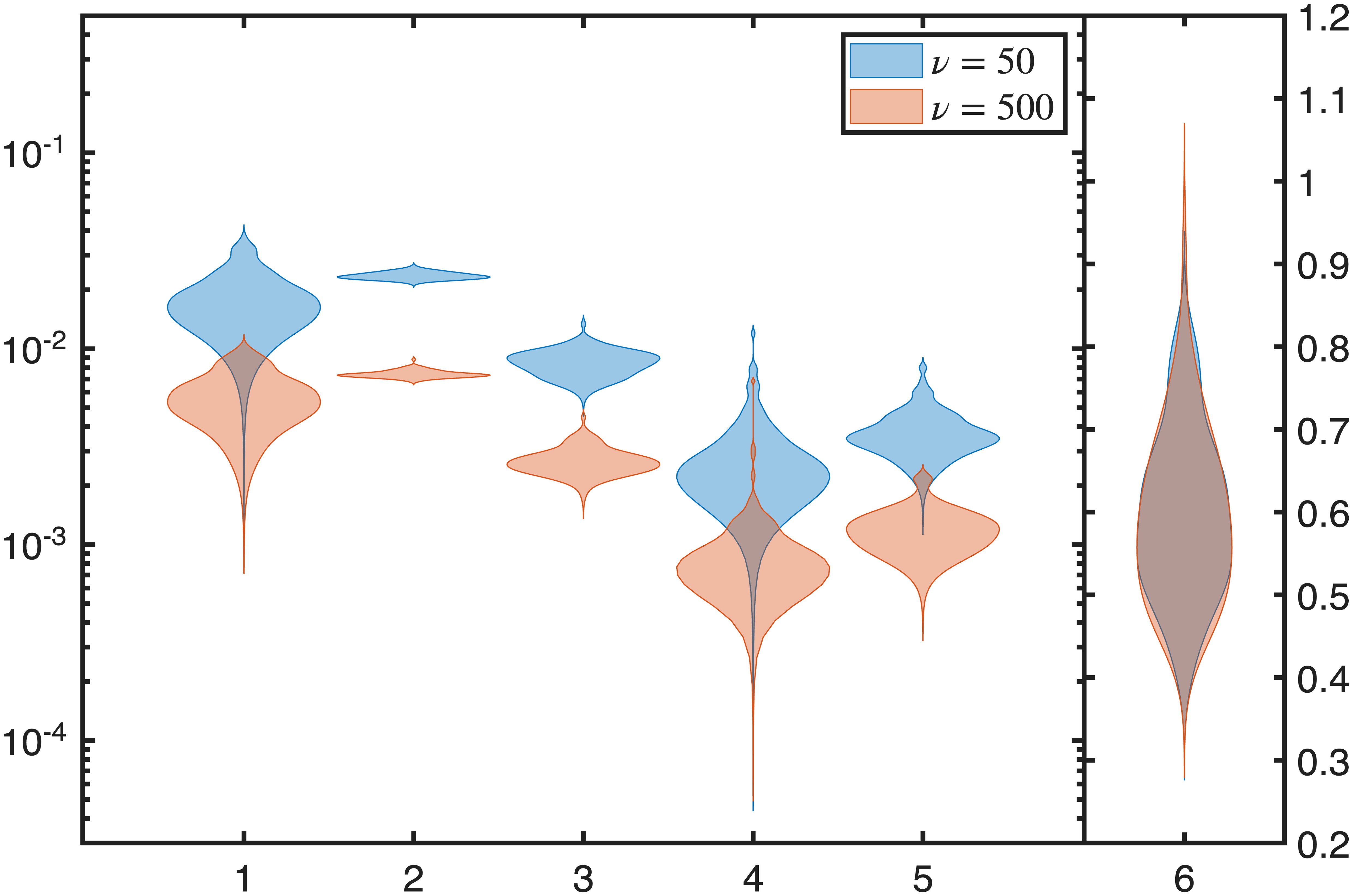}
\caption{Statistics for some key error quantities as obtained by running algorithm \ref{alg:main} on 100 independent isotropic parameter choices on the two-dimensional problem. Columns 1 to 5 are on a logarithmic scale and correspond to  $\|\mathbf{Y}^{-1} - \bar{\mathbf{Y}}^{-1}\| / \|\mathbf{Y}^{-1}\|$, $\|\mathbf{Y} - \bar{\mathbf{Y}}\|_\mathrm{F} / \|\mathbf{Y}\|_\mathrm{F}$, $\|\mathbf{Y} - \mathbf{Y}_B \|_\mathrm{F} / \|\mathbf{Y}\|_\mathrm{F}$, $\|\bar{\mathbf{w}}_\theta - \mathbf{w}\| / \|\mathbf{w}\|$, and $\|\bar{\mathbf{w}} - \mathbf{w}\| / \|\mathbf{w}\|$. The data in column 6 with linear scale refer to $c_V(\mathbf{P})$. The top panels show the variation of these errors at moderately small coefficients of variation at small $(\nu=10)$ and larger $(\nu=100)$ ensemble sizes, whilst those below refer to larger coefficients of variation, with $(\nu=50)$ and $(\nu=500)$. The panels on the left are with $\mu=1$ and those on the right with $\mu=16$. The relative error of our low-variance estimator $\bar{\mathbf{w}}_\theta$ is always smaller than that of the multi-sketch average estimator $\bar{\mathbf{w}}$, and this advantage is more pronounced in the small coefficient of variation regime.} 
\label{fig:2Dresults}
\end{figure}

\begin{figure}
\centering
\includegraphics[width=0.32\linewidth, trim=45 0 60 0]{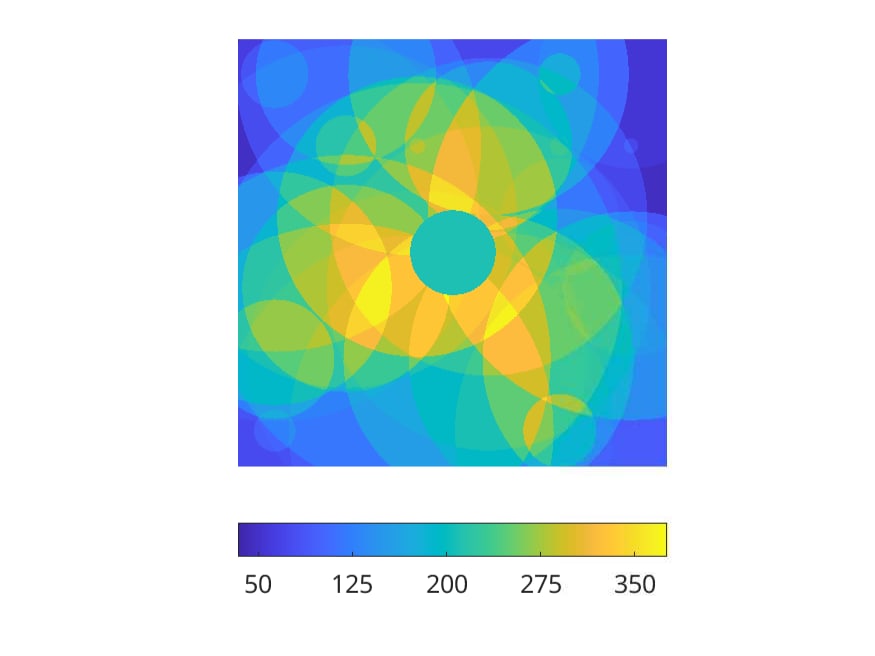}
\includegraphics[width=0.32\linewidth, trim=45 0 60 0]{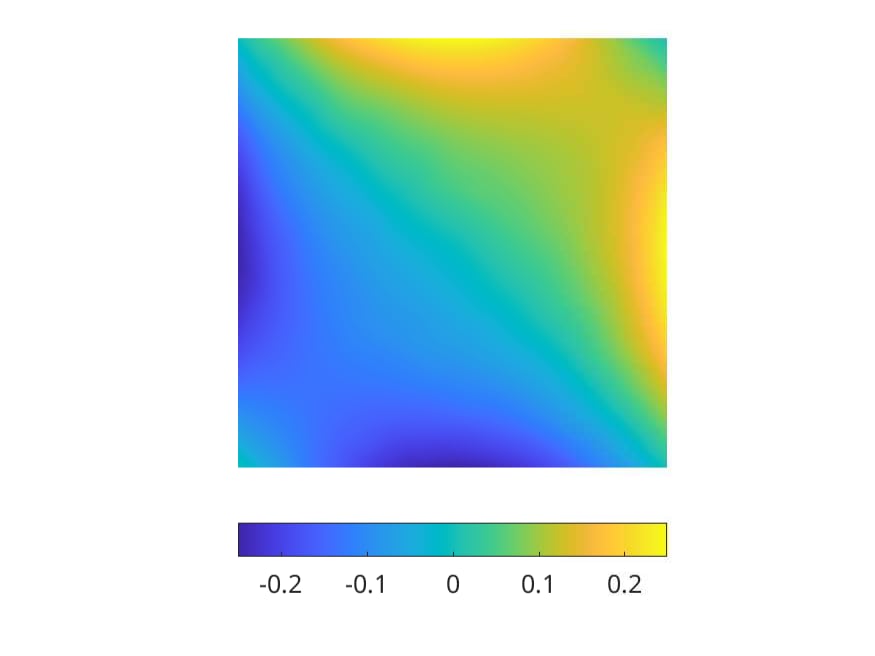}
\includegraphics[width=0.32\linewidth, trim=45 0 60 0]{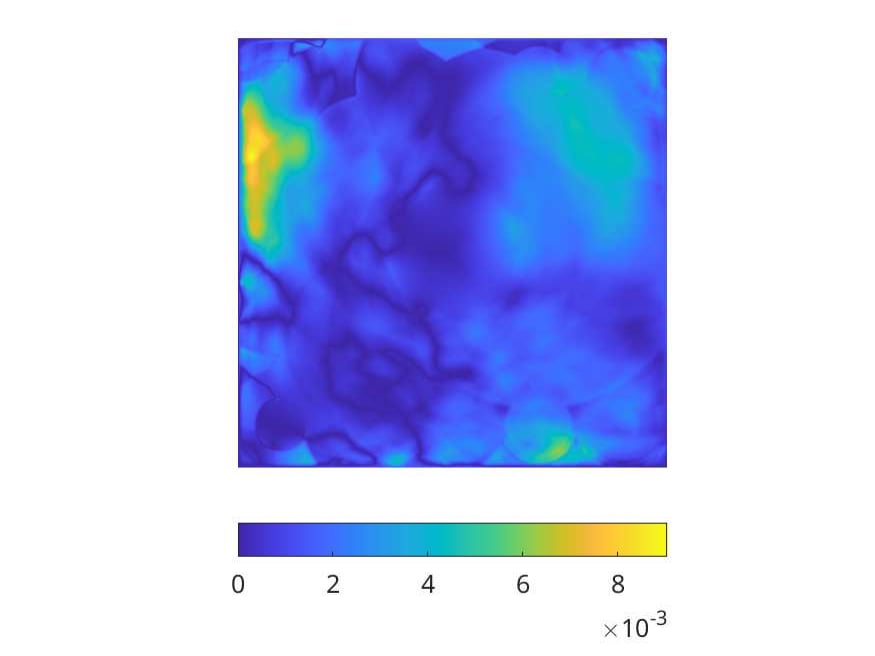}
\caption{Images of the spatial profile of an isotropic parameter field from the first set of tests with $c_V(\mathbf{P})=0.40$, $\mu=1$ and $\nu=100$ (left), its deterministic projected FEM solution $\mathbf{w}$ (middle) and the absolute pointwise error $|\mathbf{w} - \bar{\mathbf{w}}_\theta|$ from its sketched solution (right). Notice that the solution is spatially smooth and the absolute error in the sketched solution is predominantly very small apart from a few isolated areas.}
\label{fig:imageslowCov}
\end{figure}

\subsection{Three-dimensional example}

For the tests in three dimensional FEM we consider a spherical domain $\Omega$ of unit radius centred at the origin, with a fixed forcing term and boundary conditions expressed in spherical coordinates as  
$$
f(\varrho,\vartheta, \varphi) = \begin{cases}
 1 & \text{if} \quad \varrho \leq 0.15 \cos \bigl ( 3 (\vartheta + \frac \pi 3) \bigr ) \, \cos \bigl ( 2(\varphi + \frac \pi 2) \bigr ) \\
 0 & \text{otherwise}
\end{cases}, \quad u^\partial(1,\vartheta,\varphi) = 0. 
$$
The domain is discretised into a mesh of $n_e=33959932$ linear tetrahedral elements connecting $n_n =5768160$ nodes of which $m=14806$ are on the boundary, resulting into a FEM model with $n=5753354$ degrees of freedom and a large dimension $N=3 n_e \approx 102$ million. Aside the higher physical and numerical dimension of this setting we also seek to explore the performance of our approach in dealing with anisotropic parameter fields. To generate the respective parameter matrices we generate a random number between 30 and 90 of possibly overlapping spherical inclusions at arbitrary locations and radii in the closure of the unit sphere. These are assigned piecewise constant anisotropic properties, yielding diagonal tensor profiles with fractional anisotropy between 0.13 and 0.88, and CoV that average at about 0.7, so somewhat higher than those used in the isotropic 2D tests.   

Similarly to the 2D case, we fix the per-sketch budget to $c = \lceil 5 s \log s \rceil$ and choose $\nu \in \{10,100\}$. For consistency with the 2D tests, we also keep $\mu \in \{1,16\}$, where the 16 sub-regions of the sphere are equal volume wedges formed by dividing the upper and lower hemisphere into 8 `identical' wedges. The results presented in figure \ref{fig:3Dresults} show a similar pattern to those of the two-dimensional tests. Interestingly however, the reduction in the Frobenius variance of the corrected sketch (column 3) from the uncorrected (column 2) is not as large, especially in the $\mu=1$ case on the left. Recall however that these tests were performed under a larger CoV regime and thus they are more comparable with the bottom panels of figure \ref{fig:2Dresults}, even though we have used smaller sketch ensemble sizes. At $\nu=10$ (blue violins) the relative errors for our improved estimator $\bar{\mathbf{w}}_\theta$ after 100 independent runs average to 0.0121 at $\mu=1$ and drop to 0.0096 for $\mu=16$, while the corresponding figure for $\bar{\mathbf{w}}$ is 0.0161. It is important to recall that the results on $\bar{\mathbf{w}}$ are with double the quoted $\nu$, i.e. $\bar{\mathbf{Y}}(\mathbf{P})$ is a $2\nu$ sketch average. In increasing the ensemble to $\nu=100$ (orange violins), the relative errors of our estimator average at 0.0036 at $\mu=1$ and down to 0.0022 when $\mu=16$, which compared to the average relative error of 0.0052 for $\bar{\mathbf{w}}$ indicate that the half-error advantage of our method is maintained for the anisotropic case as well.    

\begin{figure}
\centering
\includegraphics[width=0.49\linewidth]{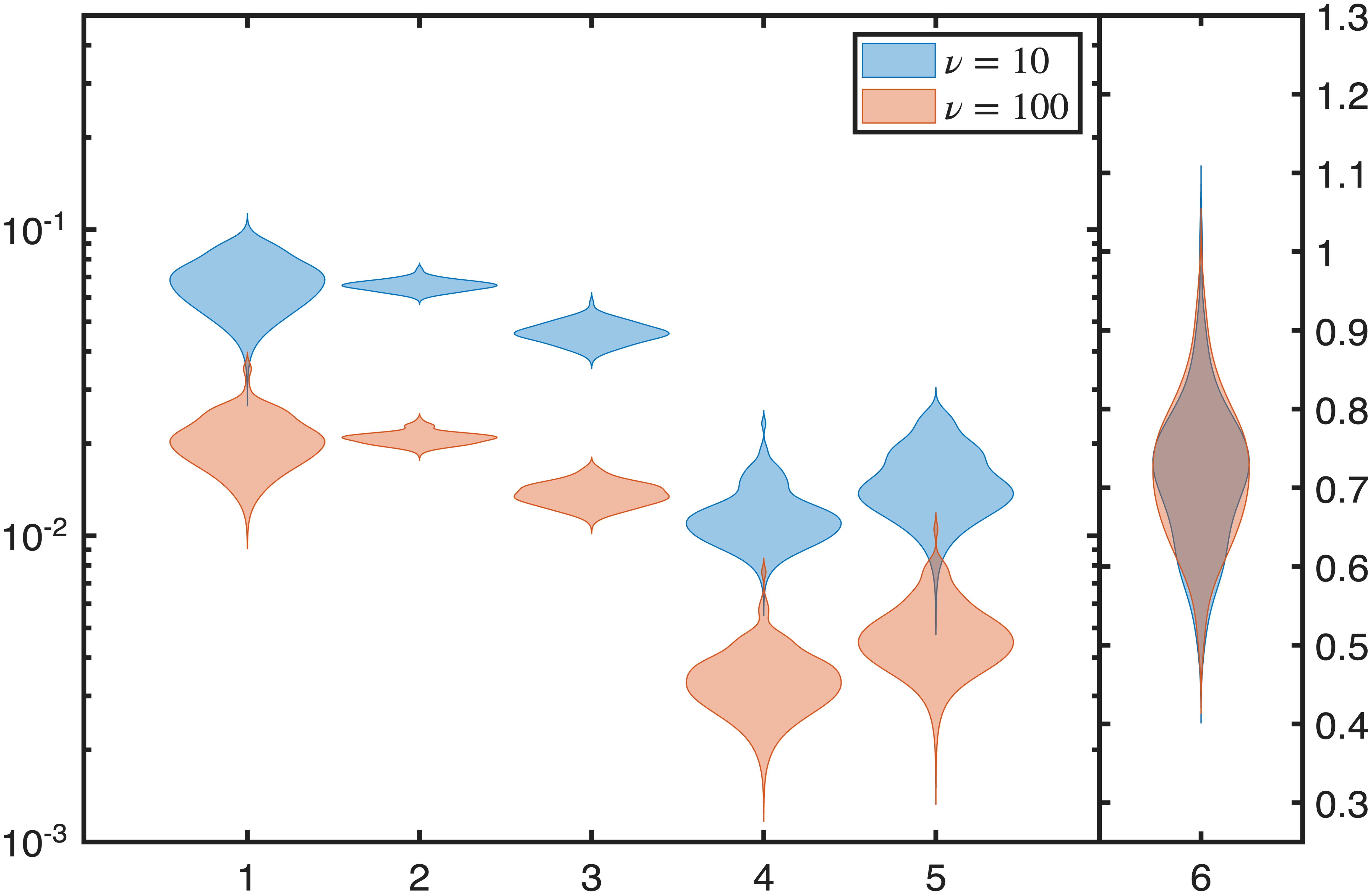}\includegraphics[width=0.49\linewidth]{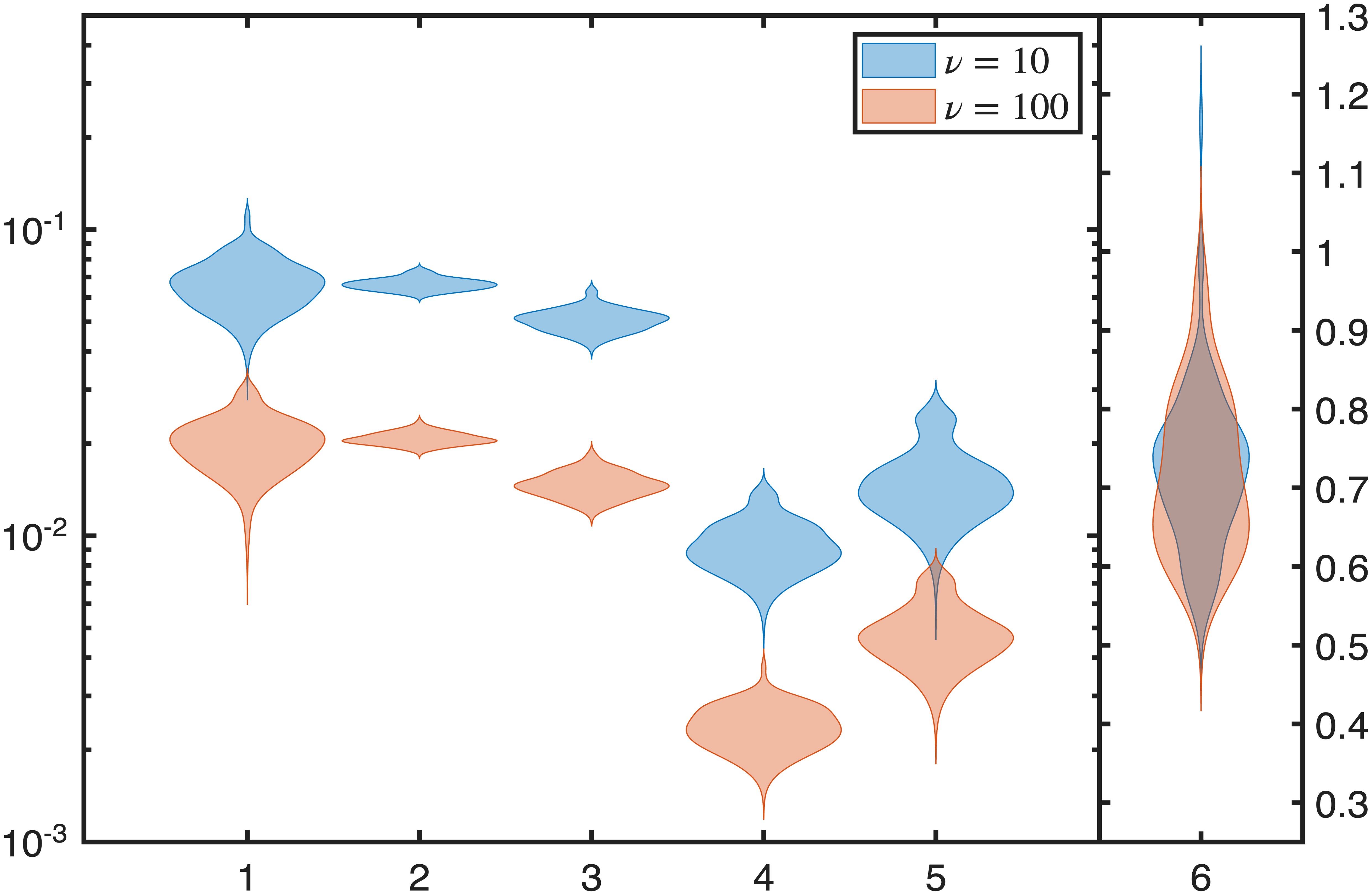}
\caption{Statistics for some key error quantities as obtained by running algorithm \ref{alg:main} on 100 independent anisotropic parameter choices on the three-dimensional problem. In each panel the columns 1 to 5 on a logarithmic scale are $\|\mathbf{Y}^{-1} - \bar{\mathbf{Y}}^{-1}\| / \|\mathbf{Y}^{-1}\|$, $\|\mathbf{Y} - \bar{\mathbf{Y}}\|_\mathrm{F} / \|\mathbf{Y}\|_\mathrm{F}$, $\|\mathbf{Y} - \mathbf{Y}_B \|_\mathrm{F} / \|\mathbf{Y}\|_\mathrm{F}$, $\|\bar{\mathbf{w}}_\theta - \mathbf{w}\| / \|\mathbf{w}\|$, and $\|\bar{\mathbf{w}} - \mathbf{w}\| / \|\mathbf{w}\|$. The data on column 6 having a linear scale refer to the coefficient of variation $c_V(\mathbf{P})$. The results follow in the trend of those of the two-dimensional problem in fig. \ref{fig:2Dresults}, in that in both cases the low-variance estimator (column 4) is on average more efficient than the multi-sketch average (column 5). 
}
\label{fig:3Dresults}
\end{figure}

\subsubsection{Timings and memory footprint} 
The computations as described in the main algorithm \ref{alg:main} and the sampling algorithm \ref{alg:fast_thinning} are performed on a Intel Xeon Gold 5120 CPU with 14 Cores, 2.2GHz Base, 19.25 MB Cache, and 380 GB RAM running Matlab 2020a.
We provide details on the most costly setup we have used, namely that of the 3D anisotropic parameter tests that involves a larger mesh, while using $\mu=16$ sub-regions for the control variates and $\nu=100$ sketch ensemble sizes with $c \approx 5300$. We report average wall-clock timings as obtained from Matlab's profiler utility, averaging over 100 sequential runs of the main algorithm. The domain clustering to sub-regions for the construction of the control variates, and the computation of the tall matrices $\mathbf{U} \in \mathds{R}^{N \times s}$ and $\mathbf{U}_\partial \in \mathds{R}^{N \times m}$ took place offline along with computing the probabilities $\boldsymbol{\eta}$ and $\tilde{\boldsymbol{\eta}}$ both in $\mathds{R}^N$ with $s=200$, $m=14806$ and $N=101879796$. Stored in whole these tall matrices take up 163 GB and 12 TB respectively, so this is best avoided as they needn't be read into the RAM in whole. Instead one can deploy segmentation and out-of-core row selection implemented using memory mapping to retrieve only the sampled rows of $\mathbf{U}$ and $\mathbf{U}_\partial$. Each of the probability vectors takes 815 MB to store. The costs for generating the parameter diagonals are not accounted for as these are assumed to be provided in our setting. The critical time to note is that for computing the projected coefficients matrices  $\mathbf{Y}(\mathbf{P})$ and $\mathbf{Z}(\mathbf{P})$ deterministically. The time to compute these are 135.4 s and 62.8 s respectively, using the built in \verb"bxfun(@times)" command, which takes about half the time for these to be computed as triple products. The time to form the control variate parameters and matrices as in the lines $4-5$ of algorithm \ref{alg:main} with $\mu=16$ and based on $500$ samples in the Monte Carlo scheme to estimate the optimal coefficients for $\mathbf{T}$ and $\mathbf{R}$ is 6.4 s. To run the geometric skipping/Bernoulli thinning algorithm \ref{alg:fast_thinning} \emph{twice} as in line 6, each time sampling the sketch and its control variates takes 30.7 s, the most costly operation of our approach. Computing the ensemble means and covariances, control variate weights, the regularisation parameter $\theta$ and the corrected sketch as per lines $7-10$ costs only 0.2 s as all these computations are bounded in dimension by $\max\{s,\nu\}$. Finally evaluating the solution of the optimisation and the answer as in lines $11-12$ takes only a few milliseconds. In conclusion for the choice of $(s,\nu,\mu)$ our algorithm takes less than half the time of computing $\mathbf{Y}(\mathbf{P})$  deterministically, using less than 1\% of the $N$ indices across the whole sampling process.         

For further context, note that to use an efficient Krylov iteration method like the preconditioned conjugate gradients algorithm \cite{GolubVanLoan2013} on the full-scale problem \eqref{fempde}, incurs the following costs: Around 135.4 s to assemble $\mathbf{A}(\mathbf{P})$ having about 85.5 million non-zero entries occupying 1.4 GB of memory and 62.9 s for $\mathbf{A}_\partial (\mathbf{P})$ with 540 thousand non-zeros taking 9.7 MB. Once the  system is assembled then it takes about 31.5 s to compute an incomplete Cholesky factor at a tolerance of $10^{-3}$ and then solving the system iteratively to a relative error of $10^{-3}$ in order to match the sketching precision, takes about 100 iterations and 81.8 s. Overall to solve the original FEM problem deterministically at a lower accuracy takes more than 5 minutes per parameter. By comparison computing the projection deterministically takes only 67 s for equivalent levels of projection approximation error but this requires prohibitive amounts of memory or communication. Our method achieves similar errors at about half the time with substantial storage and communication compared to the deterministic projection.    

\section{Conclusions}

The proposed variance-reduced RNLA framework for parametric FEM systems provides a clear improvement over the classical sketch-and-invert pipeline. For a fixed overall sampling budget, the combined use of control variates and forward-inverse sketch fusion yields solution errors that are, on average, approximately half those of the standard approach, while incurring only a marginal increase in computational complexity. This improvement is most pronounced when the parameter fields, whether isotropic or anisotropic, exhibit only moderate variation. As the fields become increasingly discontinuous, the benefit of the variance-reduction mechanism progressively declines and may eventually vanish, unless constraints on the parameters can be exploited in the construction of the control variates. The main computational bottleneck of our approach is the sampling stage itself, still the most time consuming component even after introducing geometric skipping, and therefore warrants further attention, potentially in combination with hardware acceleration. Numerical experiments demonstrate that, for parameters with a coefficient of variation up to 100\%, the proposed scheme achieves a two-fold reduction in error compared to standard randomised sketching, while requiring roughly half the computational time of deterministic multiplication. Overall, the results are sufficiently encouraging to support further development of this approach for real-time and resource-constrained scientific computing.

\appendix


\section{Proof of theorem \ref{concentave}}\label{Proofthm5.2}

We first prove the invertibility of $\bar{\mathbf{Y}}(\mathbf{P})$ conditioned on $\bar{\mathbf{Y}}(\mathbf{I}) \succ 0$ \cite{LUNG2020112933}. If $\bar{\mathbf{Y}}(\mathbf{I}) \succ 0$ then for any nonzero $\mathbf{x} \in \mathds{R}^s$, 
$\mathbf{x}^\top \bar{\mathbf{Y}}(\mathbf{I}) \mathbf{x} = \sum_{i=1}^N \bar s_{ii}^2 \bigl ( {\mathbf{u}_y}_{(i)}\mathbf{x} \bigr )^2 > 0$, where $\bar s_{ii}$ is the element of $\bar{\mathbf{S}} = \frac 1 \nu \sum_{t=1}^{\nu} \mathbf{S}_t$.
By the commutative property of the diagonal matrices $\mathbf{P}$ and $\bar{\mathbf{S}}$ we have
\begin{equation}
\mathbf{x}^\top \bar{\mathbf{Y}}(\mathbf{P}) \mathbf{x} = \sum_{i=1}^N p_{ii} \, \bar{s}_{ii}^2 \bigl ( {\mathbf{u}_y}_{(i)}\mathbf{x} \bigr )^2 > 0 \quad \Leftrightarrow \quad \bar{\mathbf{Y}}(\mathbf{P}) \succ 0, \end{equation}
since $p_{ii} >0$ and $s_{ii} \geq 0$. In effect, scaling the rows of $\mathbf{U}_y$ by $\mathbf{P}^{\frac 1 2} \succ 0$ preserves invertibility in $\bar{\mathbf{Y}}(\mathbf{P})$ provided that $\bar{\mathbf{Y}}(\mathbf{I}) \succ 0$.

Using lemma \ref{normAmI} for $0 < \epsilon <1$, we can write the error as a sum of symmetric matrices
\begin{equation}\label{thm1ap1}
\mathbf{U}_y^\top \bar{\mathbf{S}}^\top \bar{\mathbf{S}} \mathbf{U}_y - \mathbf{I} = \sum_{t=1}^{\nu} \mathbf{E}_t, \quad \mathbf{E}_t = \sum_{i=1}^{N} \mathsf{E}_{t,i}\, , \quad \mathsf{E}_{t,i} = \frac 1 \nu \Bigl ( \frac{\gamma_i^{(t)}}{\eta_i} - 1 \Bigr ) {\mathbf{u}_y}_{(i)}^\top {\mathbf{u}_y}_{(i)}, 
\end{equation}
where $\{\gamma_i^{(1)}, \ldots, \gamma_i^{(\nu)}\}_{i=1}^N$ are independent Bernoulli variables with $\mathbb{P}(\gamma_i^{(t)}=1)=\eta_i=\min\{1,c\xi_i(\mathbf{U})\}$, and ${\mathbf{u}_y}_{(i)}^\top {\mathbf{u}_y}_{(i)}$ are rank 1 $s \times s$ symmetric matrices. To bound the summands in spectral norm we have  
\begin{align}\label{thm1ap1b}
\|\mathsf{E}_{t,i}\| = \frac 1 \nu \Bigl | \frac{\gamma_i^{(t)}}{\eta_i} - 1 \Bigr | \bigl \|{\mathbf{u}_y}_{(i)}^\top {\mathbf{u}_y}_{(i)} \bigr \| = \frac 1 \nu \Bigl |\frac{\gamma_i^{(t)}}{\eta_i} - 1 \Bigr | \|{\mathbf{u}_y}_{(i)}\|^2  = \frac 1 \nu \Bigl | \frac{\gamma_i^{(t)}}{\eta_i} - 1 \Bigr | \ell_i(\mathbf{U}_y),
\end{align}
with the upper bound attained for $\eta_i < \frac 1 2$ and $\gamma_i^{(t)} =1$ for all $t \in \{\nu\}$. Introducing $\eta_i=c \xi_i(\mathbf{U}) \geq c \beta \xi_i(\mathbf{U}_y)$ and $\xi_i(\mathbf{U}_y) = s^{-1} \ell_i(\mathbf{U}_y)$ into \eqref{thm1ap1b} yields
$$
\|\mathsf{E}_{t,i}\| \leq \frac 1 \nu \Bigl ( \frac{1-\eta_i}{\eta_i} \Bigr ) \ell_i(\mathbf{U}_y) = \frac 1 \nu \Bigl ( \frac{1 - c \beta  s^{-1} \ell_i(\mathbf{U}_y)}{c \beta s^{-1} \ell_i(\mathbf{U}_y)} \Bigr ) \ell_i(\mathbf{U}_y) = \frac 1 \nu \frac{s - c \beta \ell_i(\mathbf{U}_y)}{c \beta} \leq \frac{s}{\nu c \beta}
$$
For the variance statistic consider first that
\begin{align*}
\mathbb{E}\bigl [\mathsf{E}_{t,i}^2 \bigr ] & = \frac{1}{\nu^2} \mathbb{E} \Bigl [ \Bigl (\frac{\gamma^{(t)}_i}{\eta_i} - 1 \Bigr)^2 \Bigr ] {\mathbf{u}_y}_{(i)}^\top {\mathbf{u}_y}_{(i)} {\mathbf{u}_y}_{(i)}^\top {\mathbf{u}_y}_{(i)}
= \frac{1}{\nu^2} \mathbb{E} \Bigl ( \frac{\gamma_i^2 - 2 \gamma_i \eta_i + \eta_i^2}{\eta_i^2}  \Bigr ) \ell_i(\mathbf{U}_y) {\mathbf{u}_y}_{(i)}^\top {\mathbf{u}_y}_{(i)}\\
& = \frac{1-\eta_i}{\nu^2 \eta_i} \ell_i(\mathbf{U}_y) {\mathbf{u}_y}_{(i)}^\top {\mathbf{u}_y}_{(i)}  = \frac{s - c \beta \ell_i(\mathbf{U}_y)}{\nu^2 c \beta} {\mathbf{u}_y}_{(i)}^\top {\mathbf{u}_y}_{(i)} \leq \frac{s}{\nu^2 c \beta} {\mathbf{u}_y}_{(i)}^\top {\mathbf{u}_y}_{(i)},
\end{align*}
and thus by the independence of the $\nu$ sketches the variance statistic is
$$
\bigl \|\mathbf{V}_{\mathbf{U}_y^\top \bar{\mathbf{S}}^\top \bar{\mathbf{S}} \mathbf{U}_y} \bigr \| = \Bigl \| \sum_{t=1}^{\nu} \mathbb{E}[\mathbf{E}_t^2] \Bigr \| = \nu \Bigl \| \mathbb{E}[\mathbf{E}_t^2] \Bigr \| = \nu \Bigl \| \mathbb{E}[\mathbf{E}_{t,i}^2] \Bigr \| \leq \nu \Bigl \| \frac{s}{\nu^2 c \beta} {\mathbf{u}_y}_{(i)}^\top {\mathbf{u}_y}_{(i)} \Bigr \| \leq \frac{s}{\nu c \beta}.
$$
Assembling into Bernstein's inequality for matrices we have \cite{Tropp_MatrixIneq}
\begin{align*}
\mathbb{P}\bigl (\|\mathbf{U}_y^\top \bar{\mathbf{S}}^\top \bar{\mathbf{S}} \mathbf{U}_y - \mathbf{I} \| \geq \epsilon \bigr ) \leq 2 s \exp \Bigl ( - \frac{\frac{\epsilon^2}{2}}{\frac{s}{\nu \beta c} + \frac{s \epsilon}{3 \nu \beta c}}\Bigr ) = 2 s \exp \Bigl ( - \frac{3 \nu \beta c \epsilon^2}{6 s + 2 s \epsilon} \Bigr ).
\end{align*}

\section{Proof of proposition \ref{smallc}}\label{Proofprop521}

Rank sub-additivity for the $\nu$ sketches implies
$$
\mathrm{rank} \bigl ( \bar{\mathbf{Y}}(\mathbf{P}) \bigr ) = \mathrm{rank} \bigl (\sum_{t=1}^{\nu}\hat{\mathbf{Y}}_t(\mathbf{P})\bigr )
\;\leq \;
\sum_{t=1}^{\nu} \mathrm{rank} \bigl (\hat{\mathbf{Y}}_t(\mathbf{P}) \bigr )
\;\leq \; \nu r,
$$
and thus $\mathrm{rank}(\hat{\mathbf{Y}}(\mathbf{P}))$ equals the rank of the sum of the $\nu$ sketches, which is also bounded by $s$. Letting $w_i \doteq \frac 1 \nu \sum_{t=1}^{\nu} \frac{\gamma_i^{(t)}}{\eta_i}$ allows to write the average sketch as a sum of rank 1 PSD matrices $\bar{\mathbf{Y}}(\mathbf{P})=\sum_{i \, \cdot \, w_i >0} p_{ii} \mathbf{u}_{(i)}^\top \mathbf{u}_{(i)}$, and then using the fact that the null space of the sum is the intersection of the null spaces of the summands we get
$$
\mathcal{N} ( \bar{\mathbf{Y}}(\mathbf{P})) = \bigcap_{i\, \cdot \, w_i > 0} \mathcal{N} ( \mathbf{u}_{(i)}^\top \mathbf{u}_{(i)}) = \bigcap_{i \, \cdot \, w_i > 0} \bigl \{ \mathbf{x} \, \cdot \, \mathbf{u}_{(i)}^\top \mathbf{x} = 0 \bigr \}.
$$
If for some spanning set $J_\star$ each index $i\in J_\star$ is selected at least once across the $\nu$ sketches, then $w_i>0$ for all $i\in J_\star$ and hence the rows $\{\mathbf{u}_i : i\in J_\star\}$ (which span $\mathds{R}^s$ by assumption) are present, then $\bar{\mathbf{Y}}(\mathbf{P})\succ 0$. Moreover, independence across sketches and rows implies that index $i \in \{N\}$ is not included in a given sketch with probability $(1-\eta_i)$, and thus the probability of it being not included across all $\nu$ sketches is $(1-\eta_i)^{\nu}$. Therefore,
$$
\mathbb{P}\big(\text{every } i\in J_\star \text{ appears at least once}\big)
\;=\;
\prod_{i\in J_\star}\bigl [\,1-(1-\eta_i)^{\nu}\,\bigr ].
$$

\section{Proof of remark \ref{varcovar}}\label{Proofrem53}

For any index $1 \leq h \leq s$, the diagonal element of $\bar{\mathbf{Y}}(\mathbf{P})$ satisfies 
$$
\bar{y}_{hh}(\mathbf{P})=\frac 1 \nu \sum_{t=1}^{\nu} \sum_{i=1}^N \frac{\gamma_i}{\eta_i} p_{ii} u_{ih}^2, \quad \text{and} \quad \mathbb{E}[\bar{y}_{hh}(\mathbf{P})]=y_{hh}(\mathbf{P}).
$$
The covariance of any two diagonal elements is
$$
\mathbb{C}(\bar y_{hh}(\mathbf{P}),\bar y_{kk}(\mathbf{P})) = \mathbb{E}[\bar y_{hh}(\mathbf{P})\bar y_{kk}(\mathbf{P})] - \mathbb{E}[\bar y_{hh}(\mathbf{P})] \mathbb{E}[\bar y_{kk}(\mathbf{P})] = \mathbb{E}[\bar y_{hh}(\mathbf{P})\bar y_{kk}(\mathbf{P})] - y_{hh}(\mathbf{P}) y_{kk}(\mathbf{P}).   
$$
By the independence of the $\nu$ sketches and the variables $\boldsymbol{\gamma}$ we have $\mathbb{V}(\bar{y}_{hk}(\mathbf{P}))=\frac 1 \nu \mathbb{V}(\hat{y}_{hk}(\mathbf{P}))$, and $\mathbb{C}(\bar{y}_{hh}(\mathbf{P}),\bar{y}_{kk}(\mathbf{P}))=\frac 1 \nu \mathbb{C}(\hat{y}_{hk}(\mathbf{P}), \hat{y}_{kk}(\mathbf{P}))$ so it suffices to prove the remark's claim for the individual sketches. Effectively, for a fixed $t$
\begin{align*}
\mathbb{E}[\hat{y}_{hh}(\mathbf{P})\hat{y}_{kk}(\mathbf{P})] & = \mathbb{E} \biggl [ \sum_{i=1}^N \frac{\gamma_i}{\eta_i} p_{ii} u_{ih}^2 \Bigl ( \frac{\gamma_i}{\eta_i} p_{ii} u_{ik}^2 + \sum_{j\neq i=1}^N \frac{\gamma_j}{\eta_j} p_{jj} u_{jk}^2 \Bigr ) \biggr ]\\
& = \mathbb{E} \biggl [ \sum_{i=1}^N \frac{\gamma_i^2}{\eta_i^2} p_{ii}^2 u_{ih}^2 u_{ik}^2 + \sum_{i=1}^N \sum_{j\neq i=1}^N \frac{\gamma_i \gamma_j}{ \eta_i \eta_j} p_{ii} p_{jj} u_{ih}^2 u_{jk}^2 \biggr ]\\
& = \frac 1 \nu \Bigl ( \sum_{i=1}^N \frac{1}{\eta_i} p_{ii}^2 u_{ih}^2 u_{ik}^2 + \sum_{i=1}^N p_{ii} u_{ih}^2 \sum_{j\neq i=1}^N  p_{jj} u_{jk}^2 \Bigr )\\
& = \sum_{i=1}^N \frac{1}{\eta_i} p_{ii}^2 u_{ih}^2 u_{ik}^2 + \sum_{i=1}^N p_{ii} u_{ih}^2 \bigl ( y_{kk} - p_{ii}u_{ik}^2 \bigr ) 
 = \sum_{i=1}^N \Bigl (\frac{1}{\eta_i} - 1 \Bigr ) p_{ii}^2 u_{ih}^2 u_{ik}^2 + y_{hh}y_{kk},
\end{align*}
hence $\mathbb{C}(\hat{y}_{hh},\hat{ y}_{kk}) = \sum_{i=1}^N \Bigl (\frac{1}{\eta_i} - 1 \Bigr ) p_{ii}^2 u_{ih}^2 u_{ik}^2 = \mathbb{V}(\hat{y}_{hk})$. The equation on the right can be shown easily by considering the independence of the Bernoulli $\mathbb{V}(\hat{y}_{hk})=\sum_{i=1}^N \mathbb{V} \bigl (\frac{\gamma_i}{\eta_i}p_{ii} u_{ih} u_{ik} \bigr )$.

\section{Proof of lemma \ref{VarYbarP}}\label{Prooflem54}

Splitting the total sketch variance into that of the diagonal $(h=k)$ and off-diagonal $(h \neq k)$ elements gives
$$
\mathbb{V} \bigl (\bar{\mathbf{Y}}(\mathbf{P}) \bigr ) = \frac 1 \nu \sum_{h=1}^s \mathbb{V}(\hat y_{hh}(\mathbf{P})) + \frac 1 \nu \sum_{h\neq k}^s \mathbb{V}(\hat y_{hk}(\mathbf{P})),
$$
and by symmetry we have $\hat{y}_{hk}(\mathbf{P}) = \hat{y}_{kh}(\mathbf{P}) \quad \Longleftrightarrow \quad \mathbb{V} \bigl ( \hat{y}_{hk}(\mathbf{P}) \bigr ) = \mathbb{V} \bigl ( \hat{y}_{kh}(\mathbf{P}) \bigr )$. Imparting \eqref{lemma2} from remark \ref{varcovar}, the off-diagonal sum can be developed to
$$
\mathbb{V} \bigl (\bar{\mathbf{Y}}(\mathbf{P}) \bigr ) = \frac 1 \nu \sum_{h=1}^s \mathbb{V}(\hat y_{hh}(\mathbf{P})) + \frac 1 \nu \sum_{h=1}^s \sum_{k \neq h}^s \mathbb{C} \bigl (\hat y_{hh}(\mathbf{P}), \hat y_{kk}(\mathbf{P}) \bigr )
$$
which is by definition the variance of the sum of the diagonal entries of $\bar{\mathbf{Y}}(\mathbf{P})$, hence 
\begin{align}\label{VarTrace}
\mathbb{V} \bigl (\bar{\mathbf{Y}}(\mathbf{P}) \bigr ) = \mathbb{V} \Bigl ( \sum_{h=1}^s \bar{y}_{hh}(\mathbf{P}) \Bigr ) = \mathbb{V} \Bigl ( \mathrm{Tr} \bigl ( \bar{\mathbf{Y}}(\mathbf{P}) \bigr ) \Bigr ).
\end{align}
In effect, the total variance becomes 
\begin{align*} 
\mathbb{V} \bigl (\bar{\mathbf{Y}}(\mathbf{P}) \bigr ) & = \frac 1 \nu \mathbb{V} \Bigl ( \sum_{h=1}^s \sum_{i=1}^N \frac{\gamma_i}{\eta_i}p_{ii} u_{ih}^2 \Bigr ) = \frac 1 \nu \mathbb{V} \Bigl ( \sum_{i=1}^N \frac{\gamma_i}{\eta_i}p_{ii} \ell_i(\mathbf{U}) \Bigr ) = \frac 1 \nu  \sum_{i=1}^N \Bigl ( \frac{1}{\eta_i} - 1 \Bigr ) p_{ii}^2 \ell_i(\mathbf{U})^2\\
& \leq \frac 1 \nu \sum_{i=1}^N p_{ii}^2 \ell_i(\mathbf{U}) \Bigl ( \frac s c - \ell_i(\mathbf{U}) \Bigr ),
\end{align*}
and the concentration bound in Frobenius norm \eqref{Frob_conc} follows directly from Markov's inequality.

\section{Proof of theorem \ref{theoremYPerr}}\label{Proofthm55}

From a single sketch $\hat{\mathbf{Y}}(\mathbf{P}) = \mathbf{U}^\top \mathbf{P}^{\frac 1 2} \mathbf{S}^\top \mathbf{S} \mathbf{P}^{\frac 1 2} \mathbf{U}$ of $\mathbf{Y}(\mathbf{P})=\mathbf{U}^\top \mathbf{P}\mathbf{U}$ we have
$$
\hat{\mathbf{Y}}(\mathbf{P}) - \mathbf{Y}(\mathbf{P}) = \sum_{i=1}^N \mathbf{E}_i, \quad \mathbf{E}_i = \Bigl ( \frac{\gamma_i}{\eta_i} - 1 \Bigr ) p_{ii} \mathbf{u}_{(i)}^\top \mathbf{u}_{(i)}, \quad \text{such that} \quad \mathbb{E}[\mathbf{E}_i]=0. 
$$
From the independence of $\{\gamma_i\}_{i=1}^N$ and the formula for the variance of a single sketch 
$$
\mathbf{V}_{\hat{\mathbf{Y}}}
 = \sum_{i=1}^N \Bigl ( \frac{1}{\eta_i} - 1 \Bigr ) p_{ii}^2 \ell_i(\mathbf{U}) \mathbf{u}_{(i)}^\top \mathbf{u}_{(i)},
$$
averaging over $\nu$ iid sketches, and taking the spectral norm gives the variance statistic
$$
\bigl \|\mathbf{V}_{\bar{\mathbf{Y}}} \bigr \|
 = \frac 1 \nu \Bigl  \| \sum_{i=1}^N \Bigl ( \frac{1}{\eta_i} - 1 \Bigr ) p_{ii}^2 \ell_i(\mathbf{U}) \mathbf{u}_{(i)}^\top \mathbf{u}_{(i)} \Bigr \|.
$$
To get the upper bound on the spectral norm of the error summands $\|\mathbf{E}_i\| \leq L$, note that the coefficients $\gamma_i \eta_i^{-1} - 1$ are maximised for $\gamma_i=1$ and $\eta_i < \frac 1 2$ where $\eta_i = c s^{-1} \ell_i(\mathbf{U})$. In effect, this yields
$$
\|\mathbf{E}_i\| = \Bigl \| \Bigl ( \frac{s}{c \ell_i(\mathbf{U})} - 1 \Bigr ) p_{ii} \mathbf{u}_{(i)}^\top \mathbf{u}_{(i)} \Bigr \| \leq \Bigl ( \frac{s}{c \ell_i(\mathbf{U})} - 1 \Bigr ) p_{ii} \ell_i(\mathbf{U}) \leq \frac{s}{c}\|\mathbf{P}\|_\infty. 
$$

\section{Proof of theorem \ref{thm:bernstein-spectral-ZP}}\label{Proofthm56}

To prove the tail concentration bound for the right hand side average sketch matrix consider 
$$
\bar{\mathbf{Z}}(\mathbf{P}) - \mathbf{Z}(\mathbf{P}) = \sum_{t=1}^\nu \mathbf{E}_t, \quad  \mathbf{E}_t = \sum_{i=1}^N \mathsf{E}_{t,i}, \quad \mathsf{E}_{t,i} = \Bigl ( \frac 1 \nu \frac{\tilde{\gamma}_i^{(t)}}{\tilde{\eta}_i} - 1 \Bigr ) p_{ii}\mathbf{u}_{(i)}^\top {\mathbf{u}_\partial}_{(i)},
$$
with the summands centred $\mathbb{E}[\mathsf{E}_{t,i}] = 0$. Let the symmetric dilation of the $s \times m$ summands as a $(s+m) \times (s+m)$ matrix
$$
\mathcal{E}_{t,i} \doteq \begin{bmatrix}
0 & \mathsf{E}_{t,i)}\\
\mathsf{E}_{t,i}^\top & 0
\end{bmatrix}
$$
such that $\|\sum_{t=1}^\nu \sum_{i=1}^N \mathcal{E}_{t,i}\| = \|\bar{\mathbf{Z}}(\mathbf{P}) - \mathbf{Z}(\mathbf{P})\|$. The variance of $\bar{\mathbf{Z}}(\mathbf{P})$ can also be cast as
$$
\mathbf{V}_{\bar{\mathbf{Z}}} = \sum_{i=1}^N  \mathbb{E}[\mathcal{E}_{t,i}^2] = \frac{1}{\nu} \begin{bmatrix}
\sum_{i=1}^N \mathbb{E}\bigl [\mathsf{E}_{t,i}\mathsf{E}_{t,i}^\top \bigr ] & 0\\
0 & \sum_{i=1}^N \mathbb{E} \bigl [ \mathsf{E}_{t,i}^\top \mathbf{E}_{t,i} \bigr ]
\end{bmatrix},
$$
and by the independence of the sketches 
$$
\|\mathbf{V}_{\bar{\mathbf{Z}}}\| = \frac 1 \nu \max \Bigl \{ \Bigl \|\sum_{i=1}^N \Bigl ( \frac{\tilde{\gamma}_i^{(t)}}{\tilde{\eta}_i} - 1\Bigr )p_{ii}^2 \ell_i(\mathbf{U}_\partial) \mathbf{u}_{(i)}^\top \mathbf{u}_{(i)} \Bigr \|, \Bigl \| \sum_{i=1}^N \Bigl ( \frac{\tilde{\gamma}_i^{(t)}}{\tilde{\eta}_i} - 1\Bigr )p_{ii}^2 \ell_i(\mathbf{U}) {\mathbf{u}_\partial}_{(i)}^\top {\mathbf{u}_\partial}_{(i)} \Bigr \| \Bigr \}.
$$
To get the upper bound of $\|\mathsf{E}_{t,i}\|$ we note that the sum of coefficients $\tilde{\gamma}_i^{(t)} \tilde{\eta}_i^{-1} - 1$ is maximised when $\tilde{\gamma_i}^{(t)}=1$ for all $t \in \{\nu\}$ and $\tilde{\eta}_i = \min\{1,c \zeta_i\} < \frac 1 2$. Substituting the definition of $\zeta$ from \eqref{defnzeta} yields
\begin{align*}
\|\mathsf{E}_{t,i}\| & = \frac 1 \nu \Bigl \| \Bigl ( \frac{1}{c \zeta_i} - 1 \Bigr ) p_{ii} \mathbf{u}_{(i)}^\top {\mathbf{u}_\partial}_{(i)}\Bigr \| \leq \frac{1}{\nu c} \Bigl ( \sum_{i=1}^N \ell_i(\mathbf{U}) \ell_i(\mathbf{U}_\partial) \Bigr ) \max_i \Bigl \| \frac{p_{ii}}{\ell_i(\mathbf{U}) \ell_i(\mathbf{U}_\partial)} \mathbf{u}_{(i)}^\top  {\mathbf{u}_\partial}_{(i)} \Bigr \|\\
& \leq \frac{1}{\nu c} \Bigl ( \sum_{i=1}^N \ell_i(\mathbf{U}) \ell_i(\mathbf{U}_\partial) \Bigr ) \max_i \Bigl ( \frac{p_{ii}}{\sqrt{\ell_i(\mathbf{U}) \ell_i(\mathbf{U}_\partial})} \Bigr ).
\end{align*}
Plugging these into Bernstein's inequality for rectangular zero-mean random matrices gives the result. 

\section{Sketching the right hand side vector}\label{SketchZCV}

The right hand side vector estimator $\mathbf{q}_B(\mathbf{P})$ as used in line 15 of algorithm \ref{alg:main} is evaluated as 
$$
\mathbf{q}_B(\mathbf{P}) = \bigl ( \mathbf{V} \boldsymbol{\Sigma} \bigr )^{-1} \boldsymbol{\Phi}^\top \mathbf{f} - \mathbf{Z}_B(\mathbf{P}) \boldsymbol{\Sigma}_\partial \mathbf{V}_\partial \mathbf{u}^\partial,
$$
where $\mathbf{Z}_B(\mathbf{P})$ is the control variates corrected sketch average in \eqref{cvformula}. To find the appropriate low-dimensional set of parameters $\mathbf{R} \in \mathds{L}$ and set $\bar{\mathbf{Z}}(\mathbf{R})$ as the control variates matrix for $\bar{\mathbf{Z}}(\mathbf{P})$ we set to minimise 
\begin{align*}
\mathbb{V}\bigr (\bar{\mathbf{Z}}(\mathbf{P}) - \bar{\mathbf{Z}}(\mathbf{R}) \bigr ) & = \frac 1 \nu \mathbb{V} \Bigl ( \sum_{h=1}^s \sum_{k=1}^{m} \sum_{i=1}^N \frac{\tilde{\gamma}_i}{\tilde{\eta}_i} (p_{ii} - r_{ii} )^2 u_{ih} {u_\partial}_{ik} \Bigr )\\
& = \frac 1 \nu \sum_{i=1}^N \Bigl ( \frac{1}{\tilde{\eta}_i} - 1 \Bigr ) (p_{ii} - r_{ii} )^2  (\mathbf{u}_{(i)} \mathbf{1})^2 ({\mathbf{u}_\partial}_{(i)} \mathbf{1})^2
\end{align*}
where $\mathbf{1}$ is the vector of ones in the appropriate dimension. Taking the gradient of the above to zero gives the optimal values for the degrees of freedom in the diagonal $\mathbf{R}$. 

\bibliographystyle{ieeetr} 
\bibliography{biblio} 

\end{document}